\documentclass[11pt]{article}

\usepackage[left=1in,right=1in,top=1in,bottom=1in]{geometry}

\usepackage{amsthm}
\usepackage[pagebackref,colorlinks]{hyperref}
    \hypersetup{linkcolor=black, 
        citecolor=black}
\usepackage{amsmath}

\usepackage{mathrsfs}
\usepackage[nameinlink]{cleveref}
    \crefname{ex}{Example}{Examples}
    \crefname{thm}{Theorem}{Theorems} 
    \crefname{lem}{Lemma}{Lemmas}
    \crefname{prop}{Proposition}{Propositions}
    \crefname{cor}{Corollary}{Corollaries} 
    \crefname{conj}{Conjecture}{Conjectures} 
    \crefname{defn}{Definition}{Definitions}
    \crefname{rmk}{Remark}{Remarks} 
    
	\newtheorem{thm}{Theorem}[section]
	\newtheorem{lem}[thm]{Lemma}
	\newtheorem{prop}[thm]{Proposition}

	\newtheorem*{thm*}{Theorem}
	\newtheorem*{cor*}{Corollary}
	\theoremstyle{definition} 
		\newtheorem{defn}[thm]{Definition}
		\newtheorem{ex}[thm]{Example}
    	\newtheorem{rmk}[thm]{Remark}	
    \newcommand{\df}[1]{{\bf\emph{{#1}}}}

\usepackage{tikz, graphicx}
\usepackage[margin=1.5cm]{caption}
\usepackage{subcaption}
	\usetikzlibrary{positioning}
	\usetikzlibrary{arrows}
	\usetikzlibrary{decorations.pathmorphing}
        \tikzset{%
        fwdrxn/.style={very thick, arrows={-Stealth[length=5pt,width=5pt]}},
        revrxn/.style={very thick, arrows={-Stealth[length=5pt,width=5pt,left]}},
        newt/.style={turq, opacity=0.15}
        }
        \tikzset{near start abs-right/.style={xshift=1cm}}
        \tikzset{near start abs-left/.style={xshift=-3.5cm}}
        \tikzset{near start abs-up/.style={yshift=1.5cm}}
        \tikzset{near start abs-down/.style={yshift=-1cm}}
	
\usepackage{color, xcolor}
    
	\newcommand\blue[1]{{\textcolor{blue}{#1}}}
	\definecolor{orange}{RGB}{250, 140, 0}
		
	\definecolor{turq}{RGB}{0, 160, 160}
		
	\definecolor{violet}{RGB}{164, 98, 234}
		
    \definecolor{viridisyellow}{RGB}{253,231,36}
    \definecolor{viridisyellowpale}{RGB}{239,223,81}
    \definecolor{viridisgreen}{RGB}{121,209,81}
        \definecolor{hlgreen}{RGB}{16,115,16}
    \definecolor{viridisturq}{RGB}{34,167,132}
    \definecolor{viridisblue}{RGB}{64,67,135}
    \definecolor{viridisviolet}{RGB}{68,1,84}
	
	\definecolor{ratecnst}{RGB}{172,172,172}
	
\usepackage{multirow}
\usepackage{array}

\usepackage{parskip}
\usepackage{amsfonts, amssymb, amsrefs, mathtools}
\newcommand{\eq}[1]{\begin{align*}#1\end{align*}}
	\newcommand{\eqn}[1]{\begin{align}#1\end{align}}  
	
\newcommand{\st}{\colon}                
  
\newcommand\mc[1]{\mathcal{#1}}

\newcommand\mrm[1]{\mathrm{#1}}
 
\newcommand{\rr}{\ensuremath{\mathbb{R}}}   
 
\newcommand{\zz}{\ensuremath{\mathbb{Z}}}
\newcommand{\nn}{\ensuremath{\mathbb{N}}}

\renewcommand{\epsilon}{\varepsilon}	
\renewcommand{\phi}{\varphi}			

\DeclareMathOperator{\diag}{diag}		
\DeclareMathOperator{\Span}{span}		
\newcommand{\kk}{\kappa}

\newcommand{\vv}[1]{{\boldsymbol{#1}}}  
\newcommand{\mm}[1]{\mathbf{#1}}               
\newcommand{\rrp}{\rr_{\geq}}
\newcommand{\rrpp}{\rr_{>}}
\newcommand{\zzp}{\zz_{\geq}}

\newcommand{\RR}{\ensuremath{\rightleftharpoons}}

\newcommand{\xx}{\vv x}
\newcommand{\yy}{\vv y}

\usepackage{enumitem}
\usepackage[version=4]{mhchem}
\usepackage{chemfig} \newcommand{\cf}[1]{\textsf{#1}}

\usepackage{comment}
\usepackage[colorinlistoftodos,prependcaption,textsize=footnotesize,  
    linecolor=orange, bordercolor=orange, backgroundcolor=viridisyellow, 
    textcolor=black]{todonotes}

\usepackage{authblk}\usepackage[symbol]{footmisc}
\title{
   The disguised toric locus and affine equivalence of reaction networks
}
\author[1,4]{
        Sabina J. Haque%
}\author[2,$\dagger$]{
        Matthew Satriano%
}\author[3]{
        Miruna-\c Stefana Sorea%
}\author[4]{
        Polly Y. Yu%
}
\affil[1]{\small Department of Systems Biology, Harvard University}
\affil[2]{\small Department of Pure Mathematics, University of Waterloo, Canada}
\affil[3]{\small SISSA (Scuola Internazionale Superiore di Studi Avanzati), Trieste, Italy and ULBS, Sibiu, Romania}
\affil[4]{\small NSF--Simons Center for Mathematical and Statistical Analysis of Biology, Harvard University }
\date{} 

\begin{document}
\maketitle
\footnotetext[2]{{MS was partially supported by a Discovery Grant from the National Science and Engineering Research Council of Canada and a Mathematics Faculty Research Chair.}}
\renewcommand*{\thefootnote}{\arabic{footnote}}

\begin{abstract}
    Under the assumption of mass-action kinetics, a dynamical system may be induced by several different reaction networks and/or parameters. It is therefore possible for a mass-action system to exhibit complex-balancing dynamics without being weakly reversible or satisfying toric constraints on the rate constants; such systems are called disguised toric dynamical systems. We show that the parameters that give rise to such systems are preserved under invertible affine transformations of the network. We also consider the dynamics of arbitrary mass-action systems under affine transformations, and show that there is a bijection between their sets of positive steady states, although their qualitative dynamics can differ substantially.
\end{abstract} 

\section{Introduction}
\label{sec:intro}

Mathematical models in biology and chemistry often take the form of systems of ODEs. In many applications, including mass-action kinetics in chemistry as well as population dynamics in ecology and infectious disease models, the system of ODEs arises from an interaction graph. Among such interaction network models, complex-balanced systems are some of the best understood. They are characterized by their network structure~\cite{horn1972necessary}, and their dynamical and algebraic properties are well known. Dynamically, each complex-balanced system has exactly one linearly stable positive steady state up to conservation relations~\cite{HornJackson1972, SiegelJohnston2008_notes}. In fact, the steady state is conjectured to be globally stable~\cite{Horn1974_GAC, CraciunNazarovPantea2013}. Algebraically, complex-balanced systems enjoy a toric structure: both their steady state set as well as the set of parameters required for complex-balancing are toric varieties~\cite{MR1874271,CraciunDickensteinShiuSturmfels2009}; this latter set is known as the \emph{toric locus} and is well understood in terms of the network geometry and topology.

For all the remarkable algebraic and dynamical properties of complex-balanced systems, it is natural to ask whether a mass-action system that is \emph{not} complex-balanced can enjoy these properties. This is possible since a system of differential equations is not uniquely associated to one reaction network even though a network dictates its governing dynamics, a fact sometimes referred to as \lq\lq the fundamental dogma of chemical kinetics\rq\rq~\cite{CraciunPantea2008}.  This lack of identifiability leads to the following notion: two reaction networks along with certain choices of parameters are said to be \emph{dynamically equivalent} if they generate the same dynamics. We can now ask whether a mass-action system is dynamically equivalent to, and thus shares all the properties of, a complex-balanced system~\cite{CraciunJinYu2019, SzederkenyiHangos2011a, BCS22}.

For a fixed network, its \emph{disguised toric locus} is the set of parameters such that the resulting system is dynamically equivalent to a complex-balanced system. Unlike the toric locus, the disguised toric locus is quite poorly understood. For example, computing analytically the disguised toric locus of the complete graph on $3\cf{X}$, $2\cf{X}+\cf{Y}$, $\cf{X}+2\cf{Y}$, and $3\cf{Y}$ is already non-trivial~\cite{BCS22}. 

In the present work, we aim to simplify the task of computing the disguised toric locus of a given network, by relating it to that of a different network. More precisely, we relate the disguised toric loci of \emph{affinely equivalent} reaction networks (\Cref{def:affine-trans}), since preserving collinear vertices is necessary for dynamical equivalence (see \Cref{rmk:collinear}). It was shown in \cite[Theorem 6.1]{craciun2020structure} that the toric locus is unchanged when a network is mapped under an invertible affine transformation. Here, we generalize this result: the disguised toric locus is also preserved under invertible affine transformations. As a consequence of our result, to compute the disguised toric locus of a complicated network, it suffices to compute it for an affinely equivalent, but simpler, network.

Moreover, affine transformations also preserve dynamically important structural features of networks, such as the property of being strongly endotactic; in particular, this implies that properties such as persistence and permanence are preserved by such network transformations~\cite{CraciunDickensteinShiuSturmfels2009, CraciunNazarovPantea2013}. One may further ask if other qualitative dynamics are preserved, which we consider in \Cref{sec:discussion}.

The present work can also be understood in the context of how network geometry and topology influence qualitative dynamics. There has been a longstanding tradition of relating network structure with dynamics. Examples include the existence of a positive steady state for weakly reversible networks~\cite{Boros2019}, parametrization of the steady state sets~\cite{PerezMillanDickensteinShiuConradi2012, ThomsonGunawardena2009, Johnston2014}, persistence and permanence of endotactic and strongly endoctactic networks~\cite{GopalkrishnanMillerShiu2014, Craciun2019, CraciunDeshpande2020}, establishing multistationarity~\cite{JoshiShiu2015, JoshiShiu2017, BanajiPantea2016, BanajiPantea2018, MR4413372}, and precluding multistationarity~\cite{CraciunFeinberg2005, CraciunFeinberg2006}. At times, conclusions are drawn using information from both the topology and geometry of the network~\cite{Johnston2014, JohnstonMullerPantea2019,  MullerRegensburger2012, ShinarFeinberg2010, CraciunJinYu_STN}. In this paper, we focus on network geometry as it relates to complex-balancing dynamics. 

To consider the effect of network geometry on dynamical equivalence, we further study invertible projective transformations, which also preserve collinearity. We find that such maps preserve neither dynamical equivalence, complex-balancing, nor the property of being disguised toric (\Cref{sec:discussion}). Therefore, affine transformations form the largest class of graph isomorphisms that preserve the disguised toric locus.

This work is organized as follows. We review some preliminary and relevant notions about reaction networks and mass-action systems in \Cref{sec:crn}, and explain more precisely what we mean by the geometry of a reaction network. We discuss dynamical equivalence in \Cref{sec:DE}, and complex-balanced and disguised toric systems in \Cref{sec:CB}. After introducing affine equivalence of reaction networks in \Cref{sec:main}, we prove our main result (\Cref{thm:main}) and an analogous result for detailed-balanced systems. \Cref{sec:discussion} consists of two parts. In the first half, we study the qualitative dynamics of arbitrary affinely equivalent mass-action systems, and show that there is a bijection between the sets of positive steady states, even though other interesting dynamics such as stability and multistationarity may be lost. In the second part of \Cref{sec:discussion}, we consider invertible projective transformations as generalizations of affine transformations.

\section{Preliminaries} 
\label{sec:prelim}

In this section, we summarize the relevant standard notions and notation from the theory of reaction networks. For details, we refer the reader to, for instance \cite{HornJackson1972, Feinberg1987, YuCraciun, feinberg, waage1986studies}. Throughout this work, the symbols $\rrp^n$, $\rrpp^n$ denote the sets of vectors with non-negative and positive components respectively. For $\xx \in \rrpp^n$ and $\yy \in \rr^n$, we use the multi-index notation $\xx^{\yy} \coloneqq x_1^{y_1}x_2^{y_2}\cdots x_n^{y_n}$.

\subsection{Reaction networks and mass-action systems}
\label{sec:crn} 

We model the long-term dynamical behaviour of species concentrations, using an autonomous ODE system, whose terms are dictated by a reaction network. We view reaction networks as \emph{Euclidean embedded graphs}, i.e., the network is embedded in $\rr^n$ according to the \emph{stoichiometric coefficients} attached to each vertex~\cite{Craciun2019}. This formulation is equivalent to the classical definition of a reaction network as a triple $(\mc S, \mc C, \mc R)$, where $\mc S$ is the set of species, $\mc C$ the set of complexes, and $\mc R$ the set of reactions. Complexes are formal linear combinations of species; for instance, $3\cf{X}_1+4\cf{X}_2$ is a complex. In our framework, the stoichiometric coefficients in the formal linear combination are entries in a vector $\yy_i$, living in the Euclidean space, whose  dimension is the number of species interacting in the network. We differ from the classical definition only in that we do not require the stoichiometric coefficients to be non-negative integers, but instead can be any real numbers.

\begin{defn}
A \df{reaction network} (or \df{network} for short) is a directed graph $G = (V,E)$, where  $V \subset \rr^n$ and $E \subseteq V \times V$. An edge $(\yy_i,\yy_j) \in E$ is denoted by $(i,j)$ or $\yy_i \to \yy_j$. A vertex $\yy_i$ is a \df{source vertex} if there exists another vertex $\yy_j\in V$ such that $(\yy_i,\yy_j) \in E$. The set of source vertices is denoted by $V_s$. A vertex is also called a \df{complex}, while an edge is called a \df{reaction}.
\end{defn}

The network $G$ is said to be \df{weakly reversible} if every connected component is strongly connected. Weak reversibility implies $V_s = V$, while the converse is not true in general. For example, the network $\yy_1 \to \yy_2 \RR \yy_3$ is not weakly reversible, but $V_s = V$. 

We work under the assumption of mass-action kinetics.

\begin{defn}
Let $G$ be a network, and $\vv \kk \in \rrpp^{E}$ be a vector of rate constants. The \df{mass-action system}  is the weighted directed graph $(G,\vv\kk)$ with induced dynamics on $\rrpp^n$ governed by 
\eqn{ \label{eq:mas}
    \dot{\xx} &= \vv F_{(G,\vv\kk)}(\xx) ,
\intertext{where} 
    \vv F_{(G,\vv\kk)}(\xx) &\coloneqq \sum_{(i,j) \in E} \kk_{ij} \xx^{\yy_i} (\yy_j - \yy_i). \nonumber 
}
\end{defn}

In what follows, unless otherwise specified, $\kk_{ij}$ is the rate constant of the edge $\yy_i \to \yy_j$. Occasionally, it is more convenient to order the edges; then we let $\kk_i$ be the rate constant of the $i$th edge. 

If we further suppose that $V \subset \zzp^n$ (as opposed to $V \subset \rr^n$), then the positive orthant $\rrpp^n$ is \emph{forward-invariant} under \eqref{eq:mas}, i.e., if $\xx(t)$ is a solution to the initial value problem \eqref{eq:mas} with $\xx(0) \in \rrpp^n$, then $\xx(t) \in \rrpp^n$ for all $t \geq 0$~\cite[Lemma II.1]{Sontag2001}. Since the \df{stoichiometric subspace} $S = \Span\{ \yy_j - \yy_i \st \yy_i \to \yy_j \in E\}$ contains $\dot{\xx}$, the solution $\xx(t)$ is confined to the \df{stoichiometric compatibility class} $(\xx(0) + S)\cap \rrpp^n$.

In this work we take $V \subset \rr^n$, as opposed to the more classical assumption of $V \subset \zzp^n$. In formal chemical kinetics, it is often assumed that $V \subset \zzp^n$, so that \eqref{eq:mas} can be written as $\dot{x}_i = f_i(\xx) - x_i g_i(\xx)$, where $f_i(\xx)$ and $g_i(\xx)$ are positive sums of monomials. Indeed, any such system arises from a mass-action system whose construction is given by the \lq\lq Hungarian Lemma\rq\rq~\cite{harstoth1979} with $V \subset \zzp^n$. With $V \subset \zzp^n$, the positive orthant is forward-invariant, and the solution to any initial value problem is unique.

Horn and Jackson in their seminal work~\cite{HornJackson1972} generalized it to the case $V \subset \rrp^n$, with identical dynamical results still holding at this level of generality; in particular, the positive orthant is forward-invariant. If in addition we have $V \subset (\{0\} \cup [1,\infty))^n$, then the right-hand side is Lipschitz, and the solution to any initial value problem is unique~\cite{Sontag2001}.

Here we further relax the assumption to $V \subset \rr^n$. In this case, we are only interested in the qualitative dynamics on $\rrpp^n$, and the set of positive steady states. Note that on $\rrpp^n$, any trajectory curve of \eqref{eq:mas} is also a trajectory curve of the system $\dot{\xx} = x_1^N\cdots x_n^N \cdot \vv F_{(G,\vv\kk)}(\xx)$ for $N > 1$, since the two systems are related by a non-vanishing positive scalar field. Moreover, for the system $\dot{\xx} = x_1^N\cdots x_n^N \cdot \vv F_{(G,\vv\kk)}(\xx)$,  with  sufficiently large $N \in \nn$, the entries of its vertices $\yy_i$ are positive, so its vertex set lies in $\rrp^n$. It is important to emphasize again that with $V\subset \rr^n$ (as opposed to $V \subset \rrp^n$), only the trajectories of the systems  within $\rrpp^n$ are preserved, and it is possible that a solution will leave the positive orthant in finite time. However, this more general setting of $V \subset \rr^n$ allows us to study the qualitative dynamics of a more general class of ODEs on $\rrpp^n$.

\begin{ex}
\label{ex:intro-mas}
    We illustrate the above definitions with a concrete example, and later return to this example at the end of \Cref{sec:main}. Consider the network shown in \Cref{fig:intro-mas}, with six vertices and four edges. It is clearly not weakly reversible. This network was considered in \cite[Example 5.2]{CraciunJinYu2019}, and a generalized version of it was studied in \cite[Section 6]{BCS22}.
    
    Let $\kk_i > 0$ be the rate constant of the reaction with source $\yy_i$. Then the associated dynamics is given by 
    \eq{ 
        \begin{pmatrix}
        \dot{x} \\ \dot{y} 
        \end{pmatrix}
        &= \kk_1 \begin{pmatrix*}[r]
        1 \\1 
        \end{pmatrix*} + \kk_2 x^3 \begin{pmatrix*}[r]
        -1 \\1 
        \end{pmatrix*} + \kk_3 x^3y^2 \begin{pmatrix*}[r]
        -1 \\-1 
        \end{pmatrix*} + \kk_4 y^2 \begin{pmatrix*}[r]
        1 \\-1 
        \end{pmatrix*}
        \\&= \begin{pmatrix}
        \kk_1 -\kk_2 x^3 - \kk_3 x^3y^2 + \kk_4 y^2 \\ 
        \kk_1 +\kk_2 x^3 - \kk_3 x^3y^2 - \kk_4 y^2 
        \end{pmatrix}.
    }
    Since $V_s \subset \zzp^n$, the right-hand side of the ODE system is polynomial. In this paper, we allow $V_s \subset \rr^n$ more generally. 
\begin{figure}[h!]
\centering
    \begin{tikzpicture}[scale=1]
        \draw [step=1, gray!50!white, thin] (0,0) grid (3.5,2.5);
            \draw [->, gray] (0,0)--(3.5,0);
            \draw [->, gray] (0,0)--(0,2.5);
            \node [inner sep=2pt] (1) at (0,0) {\blue{$\bullet$}};
            \node [inner sep=2pt] (2) at (3,0) {\blue{$\bullet$}};
            \node [inner sep=2pt] (4) at (0,2) {\blue{$\bullet$}};
            \node [inner sep=2pt] (3) at (3,2) {\blue{$\bullet$}};
            \node [inner sep=2pt] (5) at (1,1) {\blue{$\bullet$}};
            \node [inner sep=2pt] (6) at (2,1) {\blue{$\bullet$}};
            
            \node at (0,0) [left] {$\yy_1$};
            \node at (0,2) [left] {$\yy_4$};
            \node at (3,0) [above right] {$\yy_2$};
            \node at (3,2) [right] {$\yy_3$};
            \node at (1,1) [above] {\,\,$\yy_5$};
            \node at (2,1) [above] {$\yy_6$\,\,};
            
            \draw [-{stealth}, thick, blue, transform canvas={xshift=0ex, yshift=0ex}] (1)--(5) ;
            \draw [-{stealth}, thick, blue, transform canvas={xshift=-0ex, yshift=-0ex}] (4)--(5) ;
            
            \draw [-{stealth}, thick, blue, transform canvas={xshift=0ex, yshift=0ex}] (2)--(6) ;
            \draw [-{stealth}, thick, blue, transform canvas={xshift=-0ex, yshift=-0ex}] (3)--(6) ;
	\end{tikzpicture}
    \caption{A reaction network in $\rr^2$. See \Cref{ex:intro-mas} for its associated dynamics.}
    \label{fig:intro-mas}
\end{figure}
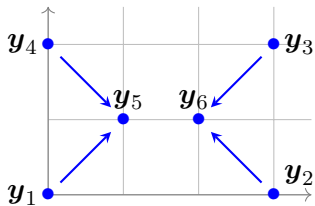 
\end{ex}

\subsection{Dynamical equivalence}
\label{sec:DE}

A mass-action system $(G,\vv\kk)$ determines a system of ODEs by \eqref{eq:mas}. However, many different mass-action systems may give rise to the same system of ODEs~\cite{CraciunPantea2008}. Such mass-action systems are said to be dynamically equivalent (see \Cref{def:dynEquivalence} below). 

Let $G = (V,E)$ and $G' = (V',E')$ be two networks in $\rr^n$. We do not assume either graph is weakly reversible. Let $V_s$ and $V'_s$ denote the sets of source vertices for the two networks respectively. The mass-action systems $(G,\vv\kk)$ and $(G', \vv\kk')$ have the same system of ODEs if $\vv F_{(G,\vv\kk)}(\xx) = \vv F_{(\vv G',\vv\kk')}(\xx)$. In expanded form, we can rearrange the right-hand sides by the distinct generalized monomials $\xx^{\yy_i}$ that appear, i.e., by sorting by their source vertices. Hence $(G,\vv\kk)$ is dynamically equivalent to $(G',\vv\kk')$ if and only if 
\eq{ 
    \sum_{\yy_i \in V_s} \xx^{\yy_i} \sum_{\yy_j \in V} \kk_{ij} (\yy_j - \yy_i) 
    = \sum_{\yy_i \in V_s'} \xx^{\yy_i} \sum_{\yy'_j \in V'} \kk'_{ij} (\yy_j' - \yy_i).
}
Since the real-valued functions $\{\xx^\yy \st \yy\in V_s\cup V_s{}'\}$ are linearly independent, we arrive at the following formal definition of dynamical equivalence~\cite{CraciunJinYu2019}. 

\begin{defn}\label{def:dynEquivalence}
Two mass-action systems $(G,\vv\kk)$ and $(G',\vv\kk')$ are \df{dynamically equivalent} if for every $\yy_i \in V_s \cup V_s'$, we have 
\eqn{\label{eq:DE}  
    \sum_{\yy_j \in V} \kk_{ij} (\yy_j - \yy_i) = \sum_{\yy'_j \in V'} \kk'_{ij} (\yy_j' - \yy_i), 
}
where by convention $\kk_{ij} = 0$ if $(\yy_i,\yy_j) \not\in E$, and similarly $\kk'_{ij} = 0$ if $(\yy_i,\yy'_j) \not\in E'$. 
\end{defn}

Since dynamical equivalence is controlled by the weighted sum of vectors originating from each source vertex $\yy_i \in V_s$ with respect to a realization $(G,\vv\kk)$, we give such sums a name.

\begin{defn}
The \df{net reaction vector} of $\yy_i \in V_s$ in a mass-action system $(G,\vv\kk)$ is 
\eqn{
    \vv w_i \coloneqq \sum_{\yy_j \in V} \kk_{ij} (\yy_j - \yy_i). 
}
If $\yy_i \not\in V_s$, then by default $\vv w_i = \vv 0$. 
\end{defn}

\begin{figure}[h!]
\centering
\begin{subfigure}[b]{0.23\textwidth}\centering 
    \begin{tikzpicture}[scale=1.75]
        \draw [step=1, gray!50!white, thin] (0,0) grid (1.25,1.25);
        \node at (0,1.35) {};
            \draw [->, gray] (0,0)--(1.25,0);
            \draw [->, gray] (0,0)--(0,1.25);
            \node [inner sep=2pt] (1) at (0,0) {\blue{$\bullet$}};
            \node [inner sep=2pt] (2) at (1,0) {\blue{$\bullet$}};
            \node [inner sep=2pt] (3) at (1,1) {\blue{$\bullet$}};
            \node [inner sep=2pt] (4) at (0,1) {\blue{$\bullet$}};
            \node [inner sep=2pt] (5) at (0.5,0.5) {\blue{$\bullet$}};
            
            \node at (0,0) [left] {$\vv y_1$};
            \node at (1,0) [below] {$\vv y_2$};
            \node at (1,1) [above] {$\vv y_3$};
            \node at (0,1) [left] {$\vv y_4$};
            \node at (0.5,0.5) [left] {$\vv y_5$\,\,};
            
            \draw [-{stealth}, thick, blue, transform canvas={xshift=-0.25ex, yshift=0.25ex}] (1)--(5) ;
            \draw [-{stealth}, thick, blue, transform canvas={xshift=0.25ex, yshift=-0.25ex}] (5)--(1) ;
            \draw [-{stealth}, thick, blue, transform canvas={xshift=-0.25ex, yshift=0.25ex}] (5)--(3) ;
            \draw [-{stealth}, thick, blue, transform canvas={xshift=0.25ex, yshift=-0.25ex}] (3)--(5) ;
            
            \draw [-{stealth}, thick, blue, transform canvas={}] (2)--(5) ;
            \draw [-{stealth}, thick, blue, transform canvas={}] (4)--(5) ;
    \end{tikzpicture}
    \caption{ \,\,$G_\mrm{a}$}
    \label{fig:DE-notWR}
\end{subfigure}
\begin{subfigure}[b]{0.23\textwidth}\centering 
    \begin{tikzpicture}[scale=1.75]
        \draw [step=1, gray!50!white, thin] (0,0) grid (1.25,1.25);
            \draw [->, gray] (0,0)--(1.25,0);
            \draw [->, gray] (0,0)--(0,1.25);
            \node [inner sep=2pt] (1) at (0,0) {\blue{$\bullet$}};
            \node [inner sep=2pt] (2) at (1,0) {\blue{$\bullet$}};
            \node [inner sep=2pt] (3) at (1,1) {\blue{$\bullet$}};
            \node [inner sep=2pt] (4) at (0,1) {\blue{$\bullet$}};
            
           \node at (0,0) [left] {\phantom{$\vv v_1$}};
           \node at (1,0) [below] {\vphantom{$\vv v_2$}};
            
            \draw [-{stealth}, thick, blue, transform canvas={xshift=-0.25ex, yshift=0.25ex}] (1)--(3) ;
            \draw [-{stealth}, thick, blue, transform canvas={xshift=0.25ex, yshift=-0.25ex}] (3)--(1) ;
            \draw [-{stealth}, thick, blue, transform canvas={xshift=-0.25ex, yshift=-0.25ex}] (2)--(4) ;
            \draw [-{stealth}, thick, blue, transform canvas={xshift=0.25ex, yshift=0.25ex}] (4)--(2) ;
    \end{tikzpicture}
    \caption{ \,\,$G_\mrm{b}$}
    \label{fig:DE-b}
\end{subfigure}
\begin{subfigure}[b]{0.23\textwidth}\centering 
    \begin{tikzpicture}[scale=1.75]
        \draw [step=1, gray!50!white, thin] (0,0) grid (1.25,1.25);
            \draw [->, gray] (0,0)--(1.25,0);
            \draw [->, gray] (0,0)--(0,1.25);
            \node [inner sep=2pt] (1) at (0,0) {\blue{$\bullet$}};
            \node [inner sep=2pt] (2) at (1,0) {\blue{$\bullet$}};
            \node [inner sep=2pt] (3) at (1,1) {\blue{$\bullet$}};
            \node [inner sep=2pt] (4) at (0,1) {\blue{$\bullet$}};
            
           \node at (0,0) [left] {\phantom{$\vv v_1$}};
           \node at (1,0) [below] {\vphantom{$\vv v_2$}};
            
            \draw [-{stealth}, thick, blue, transform canvas={xshift=-0.25ex, yshift=0.25ex}] (1)--(3) ;
            \draw [-{stealth}, thick, blue, transform canvas={xshift=0.25ex, yshift=-0.25ex}] (3)--(1) ;
            \draw [-{stealth}, thick, blue, transform canvas={xshift=-0.25ex, yshift=-0.25ex}] (2)--(4) ;
            \draw [-{stealth}, thick, blue, transform canvas={xshift=0.25ex, yshift=0.25ex}] (4)--(2) ;
            
            \draw [-{stealth}, thick, blue, transform canvas={yshift=0.3ex}] (1)--(2) ;
            \draw [-{stealth}, thick, blue, transform canvas={yshift=-0.3ex}] (2)--(1) ;
            \draw [-{stealth}, thick, blue, transform canvas={yshift=0.3ex}] (4)--(3) ;
            \draw [-{stealth}, thick, blue, transform canvas={yshift=-0.3ex}] (3)--(4) ;
            
            \draw [-{stealth}, thick, blue, transform canvas={xshift=0.3ex}] (1)--(4) ;
            \draw [-{stealth}, thick, blue, transform canvas={xshift=-0.3ex}] (4)--(1) ;
            \draw [-{stealth}, thick, blue, transform canvas={xshift=0.3ex}] (2)--(3) ;
            \draw [-{stealth}, thick, blue, transform canvas={xshift=-0.3ex}] (3)--(2) ;
    \end{tikzpicture}
    \caption{ \,\,$G_\mrm{c}$}
    \label{fig:DE-c}
\end{subfigure}
\caption{Three networks with vertices as labelled in (a). For rate constants satisfying some linear constraints as shown in \Cref{ex:DE}, the corresponding mass-action systems can be made dynamically equivalent.}
\label{fig:DE}
\end{figure}
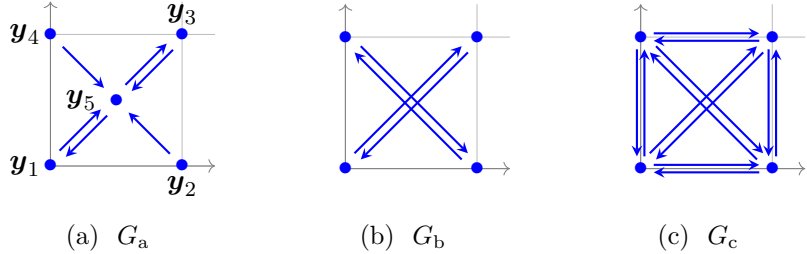

\begin{ex}
\label{ex:DE}
    Consider the three networks $G_\mrm{a}$, $G_\mrm{b}$, $G_\mrm{c}$ in \Cref{fig:DE}, with vertices $\vv y_1 = (0,0)$, $\vv y_2 = (1,0)$, $\vv y_3 = (1,1)$, $\vv y_4 = (0,1)$, and where applicable $\vv y_5 = (\frac{1}{2}, \frac{1}{2})$. Let $\alpha_{ij}$ be the rate constant of $\vv y_i \to \vv y_j$ in (a); similarly, denote rate constants for (b) as $\beta_{ij}$, and for (c) as $\gamma_{ij}$. 
    
    For $(G_\mrm{a}, \vv \alpha)$ to be dynamically equivalent to $(G_\mrm{b}, \vv\beta)$, we require that 
    \begin{gather*} 
        \alpha_{15} \begin{pmatrix*}[r]  1/2 \\ 1/2 \end{pmatrix*}= \beta_{13} \begin{pmatrix*}[r]  1 \\ 1 \end{pmatrix*}, \quad 
        \alpha_{25} \begin{pmatrix*}[r]  -1/2 \\ 1/2 \end{pmatrix*}= \beta_{24} \begin{pmatrix*}[r]  -1 \\ 1 \end{pmatrix*}, 
        \\
        \alpha_{35} \begin{pmatrix*}[r]  -1/2 \\ -1/2 \end{pmatrix*}= \beta_{31} \begin{pmatrix*}[r]  -1 \\ -1 \end{pmatrix*}, \quad 
        \alpha_{45} \begin{pmatrix*}[r]  1/2 \\ -1/2 \end{pmatrix*}= \beta_{42} \begin{pmatrix*}[r]  1 \\ -1 \end{pmatrix*}, 
        \\
        \alpha_{53} \begin{pmatrix*}[r]  1/2 \\ 1/2 \end{pmatrix*} + \alpha_{51}  \begin{pmatrix*}[r]  -1/2 \\ -1/2 \end{pmatrix*} =  \begin{pmatrix*}[r]  0 \\ 0 \end{pmatrix*}. 
    \end{gather*}
    In other words, $(G_\mrm{a}, \vv \alpha)$ is dynamically equivalent to $(G_\mrm{b}, \vv\beta)$ if and only if
    \eq{ 
        \alpha_{15} =2 \beta_{13}, \quad  \alpha_{25} = 2\beta_{24}, 
        \quad 
        \alpha_{35} = 2\beta_{31} \quad  \alpha_{45} = 2\beta_{42}, 
        \quad 
        \alpha_{53} - \alpha_{51} = 0.  
    }
    In the case when they are dynamically equivalent, the associated dynamical system is 
    \eq{ 
        \begin{pmatrix} \dot{x} \\ \dot{y} \end{pmatrix} 
        =  \begin{pmatrix*}[r]  \beta_{13} \\ \beta_{13} \end{pmatrix*} + 
        x \begin{pmatrix*}[r]  -\beta_{24} \\ \beta_{24} \end{pmatrix*} + 
        xy \begin{pmatrix*}[r]  -\beta_{31} \\ -\beta_{31} \end{pmatrix*} + 
        y \begin{pmatrix*}[r]  \beta_{42} \\ -\beta_{42} \end{pmatrix*} .
    }
    That $(G_\mrm{a}, \vv \alpha)$ and $(G_\mrm{b}, \vv\beta)$ are dynamically equivalent can be readily seen from \Cref{fig:DE-notWR,fig:DE-b}. The reaction vectors are scaled to double in length, while the rate constants are scaled accordingly.
    
    By a similar calculation, for $(G_\mrm{b}, \vv\beta)$ and $(G_\mrm{c}, \vv\gamma)$ to be dynamically equivalent, the rate constants must satisfy
    \begin{gather*}
        \begin{pmatrix*}[r]  \beta_{13} \\ \beta_{13} \end{pmatrix*} = \begin{pmatrix*}[r]  \gamma_{13} + \gamma_{12} \\ \gamma_{13} + \gamma_{14} \end{pmatrix*}, 
        \quad 
        \begin{pmatrix*}[r]  -\beta_{24} \\ \beta_{24}  \end{pmatrix*} = 
        \begin{pmatrix*}[r]  -\gamma_{24} - \gamma_{21} \\ \gamma_{24}  + \gamma_{23} \end{pmatrix*}, 
        \\[3pt] 
        \begin{pmatrix*}[r]  - \beta_{31}  \\ - \beta_{31}  \end{pmatrix*} = 
        \begin{pmatrix*}[r]  -\gamma_{31}- \gamma_{34} \\ - \gamma_{31} - \gamma_{32}  \end{pmatrix*}, 
        \quad 
        \begin{pmatrix*}[r]  \beta_{42} \\ -\beta_{42} \end{pmatrix*} = 
        \begin{pmatrix*}[r]  \gamma_{42} + \gamma_{43} \\ -\gamma_{42} -\gamma_{41} \end{pmatrix*}. 
    \end{gather*}
    Finally, we note that two dynamically equivalent mass-action systems may share the same network structure. For example, if we let $(G_\mrm{c},\vv\gamma)$ and $(G_\mrm{c},\vv\mu)$ where $\gamma_{2j} = \mu_{2j}$, $\gamma_{3j} = \mu_{3j}$, $\gamma_{4j} = \mu_{4j}$, and $\mu_{12} = \mu_{14} = \mu_{13} = \frac{\gamma_{13}}{2}$, then $(G_\mrm{c},\vv\gamma)$ and $(G_\mrm{c},\vv\mu)$ are dynamically equivalent. Because $\vv\gamma \neq \vv \mu$, we say $(G_\mrm{c},\vv\gamma)$ and $(G_\mrm{c},\vv\mu)$ are two different realizations.
\end{ex}

\subsection{Complex-balanced and disguised toric dynamical systems}
\label{sec:CB}

Complex-balanced systems are a class of mass-action systems which enjoy remarkable algebraic and stability properties. Many of their dynamical properties have been known since the seminal work of Horn and Jackson~\cite{HornJackson1972}, who intended complex-balancing to be a generalization of \emph{detailed-balancing} from thermodynamics. Later, the algebraic and combinatorial structure of complex-balanced systems was studied in \cite{CraciunDickensteinShiuSturmfels2009}. Informally, complex-balancing captures the state when the in-flux balances the out-flux at every vertex (or at every complex, hence the term complex-balancing). 

\begin{defn}
\label{def:CB} 
A mass-action system $(G,\vv\kk)$ is \df{complex-balanced} if there exists a positive steady state $\xx \in \rrpp^n$ such that for every vertex $\yy_i \in V$, the following equality holds:
\eqn{ \label{eq:CB}
    \sum_{\yy_j \in V} \kk_{ij} \xx^{\yy_i} = \sum_{\yy_j \in V} \kk_{ji} \xx^{\yy_j}.
}
\end{defn}

Although a complex-balanced system is defined to be one admitting a steady state satisfying \eqref{eq:CB}, once $(G,\vv\kk)$ has one complex-balanced steady state, then \emph{all} its positive steady states also satisfy \eqref{eq:CB}~\cite[Theorem 6A]{HornJackson1972}. This justifies calling $(G,\vv\kk)$ a complex-balanced system. 

Complex-balanced systems are remarkably stable~\cite[Theorem 2.3]{YuCraciun}. If $\xx^*$ is a complex-balanced steady state, the free energy function for detailed-balanced systems in thermodynamics
\eq{ 
    V(\xx) = \sum_{i=1}^n x_i \ln (x_i - x_i^* -1) 
}
is also a Lyapunov function for the complex-balanced system on all of $\rrpp^n$. On each stoichiometric compatibility class, the unique minimum of $V$ is a complex-balanced steady state, which is also linearly stable~\cite{SiegelJohnston2008_notes, BorosMullerRegensburger2020} and conjectured to be globally stable within its stoichiometric compatibility class~\cite{Horn1974_GAC, CraciunNazarovPantea2013}. The latter statement, known as the Global Attractor Conjecture in reaction network theory, is proved only for several special cases; for example, strongly connected networks~\cite{Anderson2011,BorosHofbauer2020}, strongly endotactic networks~\cite{GopalkrishnanMillerShiu2014}, networks in $\rr^2$~\cite{CraciunNazarovPantea2013}, and networks with three-dimensional stoichiometric subspaces~\cite{Pantea2012}.

Complex-balanced systems enjoy desirable algebro-combinatorial properties, and a surprising connection to toric geometry; thus they are also called \emph{toric dynamical systems}~\cite{CraciunDickensteinShiuSturmfels2009}. In particular, their positive steady state set admits a monomial parametrization, hence toric. Moreover, the \df{toric locus} $\mc K(G)$ of a reaction network $G$ is the set of positive $\vv\kk$ for which $(G,\vv\kk)$ is complex-balanced. After a change of coordinates, this locus is defined by binomial equations~\cite[Theorem 9]{CraciunDickensteinShiuSturmfels2009}, hence again, toric. Thanks to their strong computational and combinatorial properties, toric varieties are a class of fundamental and computationally tractable objects in algebraic geometry, see for instance \cite[page 115]{michalek2021invitation}. Projective toric varieties are described by polytopes, allowing one to understand their geometry through combinatorial methods. Since toric ideals are defined by binomial equations, they are advantageous from a computer algebra viewpoint, such as when using the Macaulay2 software~\cite{M2}. 

A necessary condition for complex-balancing is weak reversibility~\cite{horn1972necessary}, so if $G$ is not weakly reversible, $\mc K(G)=\varnothing$; for example see \Cref{fig:DE-notWR} when $\kk_{25}$, $\kk_{45}  >0$. Although such an example is not complex-balanced, it is still possible that it is \emph{dynamically equivalent} to a complex-balanced system, and hence $(G,\vv\kk)$ enjoys the same dynamical properties as a complex-balanced system. The corresponding set in parameter space was defined in \cite[Definition 2.2]{BCS22}. 
\begin{defn}
\label{def:disguised-toric}
The \df{disguised toric locus} of a reaction network $G$ is the set 
\eq{ 
    \widehat{\mc K}(G) \coloneqq \{ \vv\kk \in \rrpp^E  \st (G,\vv\kk) \text{ is dynamically equivalent to a complex-balanced system}\}.
}
If $\widehat{\mc K}(G) \neq \varnothing$, we say $G$ is \df{disguised toric}. 
\end{defn}

In other words, $\vv\kk \in \widehat{\mc K}(G)$ if and only if there exist a network $G'$ and a vector of rate constants $\vv\kk' \in \rrpp^{E'}$ such that $(G',\vv\kk')$ is complex-balanced, and $(G,\vv\kk)$ is dynamically equivalent to $(G',\vv\kk')$. By definition, $\mc K(G) \subseteq \widehat{\mc K}(G)$.

In principle, not only can the network topology change when searching for a complex-balanced realization, but also the set of vertices. In practice, however, if $(G,\vv\kk)$ is disguised toric, then a complex-balanced realization $(G',\vv\kk')$ can be found using only the source vertices of $G$~\cite[Theorem 4.7]{CraciunJinYu2019}. 

Checking whether a particular numerical vector $\vv\kk$ lies in the disguised toric locus $\widehat{\mc K}(G)$ is a linear feasibility problem, see \cite{SzederkenyiHangos2011a} and references within for an algorithm. This is much simpler than determining all of $\widehat{\mc K}(G)$. An algorithm based on quantifier elimination for computing the entire disguised toric locus $\widehat{\mc K}(G)$ is available in \cite{BCS22}. The latter might be  computationally expensive in some cases. In instances when one knows $\widehat{\mc K}(G)$ for some $G$, we can ask whether $\widehat{\mc K}(G)$ can be leveraged to study the disguised toric locus of some other network. In the next section, we answer this question when the networks are related by an invertible affine transformation.

\section{Affine equivalence of networks}
\label{sec:main}

In \cite{craciun2020structure}, it was shown that the toric locus is preserved under invertible affine transformations. Here we prove that the same is true for the disguised toric locus. We first define what is meant by transforming a reaction network by such a transformation. 

Recall that an \emph{invertible affine transformation} $\mm A \colon \rr^n \to \rr^n$ is one such that there exists an invertible linear transformation $\mm M \colon \rr^n \to \rr^n$ and a vector $\vv b \in \rr^n$ such that $\mm A (\vv y)  = \mm M \vv y + \vv b$ for all $\vv y \in \rr^n$. 

\begin{defn}
\label{def:affine-trans} 
Let $G = (V,E)$ be a network in $\rr^n$, and let $\mm A \colon \rr^n \to \rr^n$ be an invertible affine transformation. Let $\mm A(V) \coloneqq \{ \mm A(\yy) \st \yy \in V\}$ and $\mm A(E) \coloneqq \{ \mm A(\yy_i) \to \mm A(\yy_j) \st \yy_i \to \yy_j \in E\}$. Then the \df{image of $G$ under $\mm A$} is the graph $\mm A(G) = (\mm A(V), \mm A(E))$. 
\end{defn}

Every invertible affine transformation $\mm A$ induces an isomorphism of abstract graphs between the network $G$ and its image $\mm A(G)$, i.e., viewing the nodes and edges abstractly, forgetting about the embedding in $\rr^n$. In this work, we assume that $\mm A(G)$ inherits the rate constants from $G$ through this graph isomorphism. More precisely, two networks $G$ and $G'$ are \df{affinely equivalent} if there exists an invertible affine transformation $\mm A$ such that $G' = \mm A(G)$. Suppose $(G,\vv\kk)$ is a mass-action system where $\kk_{ij}$ is the rate constant of $\yy_i \to \yy_j$. Then for $\mm A(G)$, we let $\kk_{ij}$ be the rate constant of the corresponding reaction $\mm A(\yy_i) \to \mm A(\yy_j)$. We say the mass-action systems $(G,\vv\kk)$ and $(G',\vv\kk)$ are \df{affinely equivalent} if $G$ and $G'$ are affinely equivalent.

\begin{rmk}
\label{rmk:collinear}
    One might ask why we consider invertible affine maps. A major motivation is that as a graph isomorphism, affine maps preserve collinear vertices, which is necessary for dynamical equivalence. Consider the network in \Cref{fig:collinear-a}, which for any choice of positive rate constants is dynamically equivalent to the reversible network $\yy_1 \RR \yy_3$. However, its image under the map $(x,y) \mapsto (y, xy)$, the network in  \Cref{fig:collinear-b}, can never be dynamically equivalent to a reversible network. As we will see through examples of projective maps, only preserving collinearity is not sufficient. Thus,  invertible affine transformations are truly the natural graph isomorphisms when studying dynamically equivalent systems and disguised toric systems.
\end{rmk}

\begin{figure}[h!]
\centering
\begin{subfigure}[b]{0.3\textwidth}\centering 
    \begin{tikzpicture}[scale=1]
        \draw [step=1, gray!50!white, thin] (0,0) grid (2.25, 2.25);
        \node at (0,1.35) {};
            \draw [->, gray] (0,0)--(2.25,0);
            \draw [->, gray] (0,0)--(0,2.25);
            \node [inner sep=2pt] (1) at (0,1) {\blue{$\bullet$}};
            \node [inner sep=2pt] (2) at (1,1) {\blue{$\bullet$}};
            \node [inner sep=2pt] (3) at (2,1) {\blue{$\bullet$}};
            
            \node[opacity=0] at (0,0) [left] {$\vv y_1$};
            \node at (0,1) [left] {$\vv y_1$};
            \node at (1,1) [below] {$\vv y_2$};
            \node at (2,1) [right] {$\vv y_3$};
            
            \draw [-{stealth}, thick, blue] (1)--(2) ;
            \draw [-{stealth}, thick, blue] (3)--(2) ;
    \end{tikzpicture}
    \caption{}
    \label{fig:collinear-a}
\end{subfigure}
\begin{subfigure}[b]{0.3\textwidth}\centering 
    \begin{tikzpicture}[scale=1]
        \draw [step=1, gray!50!white, thin] (0,0) grid (2.25, 2.25);
        \node at (0,1.35) {};
            \draw [->, gray] (0,0)--(2.25,0);
            \draw [->, gray] (0,0)--(0,2.25);
            \node [inner sep=2pt] (1) at (0,0) {\blue{$\bullet$}};
            \node [inner sep=2pt] (2) at (1,1) {\blue{$\bullet$}};
            \node [inner sep=2pt] (3) at (1,2) {\blue{$\bullet$}};
            
            \node at (0,0) [left] {$\vv y'_1$};
            \node at (1,1) [below right] {$\vv y'_2$};
            \node at (1,2) [right] {$\vv y'_3$};
            
            \draw [-{stealth}, thick, blue] (1)--(2) ;
            \draw [-{stealth}, thick, blue] (3)--(2) ;
    \end{tikzpicture}
    \caption{}
    \label{fig:collinear-b}
\end{subfigure}
\caption{(a) A network that is dynamically equivalent to the reversible pair $\yy_1 \RR \yy_3$. (b) Its image under the map $(x,y) \mapsto (y, xy)$ can never be dynamically equivalent to a reversible pair.}
\label{fig:collinear}
\end{figure}
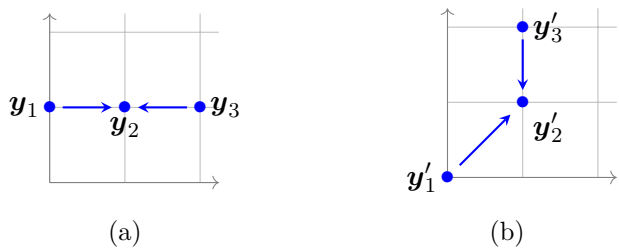

\begin{rmk}
It is not strictly necessary to restrict ourselves to the same ambient Euclidean space to define affine equivalence. If $G_1$ is a network in $\rr^{n_1}$ and $G_2$ another network in $\rr^{n_2}$, we can define $G_1$ and $G_2$ to be affinely equivalent if there is an affine equivalence as networks embedded in $\rr^{\max\{n_1,n_2\}}$ where $\rr^{n_i}$ is viewed as the first $n_i$ coordinates.
\end{rmk}

\subsection{Main result}
\label{subsec:main}

For the remainder of this section, we use the following notations. Let $\mm A (\vv y) = \mm M \vv y + \vv b$ be an invertible affine transformation, and let $(G,\vv\kk)$ be any mass-action system, with set of vertices $V$ and set of source vertices $V_s$. For each $\yy_i \in V_s$, let $\vv w_i$ denote its net reaction vector. By definition, if $\yy_i \to \yy_j \in E$, then $\kk_{ij} > 0$; if $\yy_i\to \yy_j \not\in E$,  we set $\kk_{ij} = 0$. A priori $G$ may not be weakly reversible, so it is possible that $V_s \subsetneq V$. The net reaction vector of $\yy_i \in V_s$ is 
\eqn{ \label{eq:net-dir_vect}
    \vv w_i = \sum_{\yy_j \in V} \kk_{ij} ( \yy_j -  \yy_i ),
}
where the sum is over all vertices in $V$.

\begin{lem}
\label{prop:net-dir-vect}
The net reaction vector of the vertex $\mm A(\yy_i)\in \mm A(V_s)$ in $(\mm A(G), \vv\kk)$ is $\mm M \vv w_i$, where $\vv w_i$ is the net reaction vector of the vertex $\yy_i \in V_s$ in $(G,\vv\kk)$.
\end{lem}
\begin{proof}
    On one hand, for each $\mm A(\vv y_i)$, its net reaction vector in $(\mm A(G),\vv\kk)$ is by definition, the vector
    \eq{
        \sum_{\mm A (\yy_j) \in \mm A(V)} \kk_{ij} (\mm A(\yy_j) - \mm A (\yy_i)). 
    }
    On the other hand, for each $\yy_i$, apply $\mm M$ to the representation of the net reaction vector $\vv w_i$ in \eqref{eq:net-dir_vect}: 
    \eq{ 
        \mm M\vv w_i 
        &= \sum_{\yy_j \in V} \kk_{ij} \left(\mm M \yy_j +\vv b - \mm M \yy_i - 
            \vphantom{\sum} 
            \vv b \right)
        =  \sum_{\yy_j \in V} \kk_{ij} \left(\mm A(\yy_j)  -
            \vphantom{\sum}  
            \mm A (\yy_i)  \right) .
    }
    This proves our result that $\mm M \vv w_i$ is the net reaction vector of $\mm A(\yy_i)$ in $(\mm A(G), \vv\kk)$.
\end{proof}

The following lemma highlights an important advantage of invertible affine transformations: they preserve the property of two networks being dynamically equivalent. This is an essential property, since the ODEs are the cornerstones in understanding the dynamical behaviour of the system. Recall that the main asset of using dynamical equivalence is that well-chosen distinct networks and/or distinct rate constants may give rise to the same ODE system; hence one could choose for instance networks with a more desirable combinatorial structure (such as weak reversibility), while preserving the dynamics. An example where dynamical equivalence was leveraged can be found in \cite{BCS22}. See also the discussion in \Cref{sec:discussion}.

\begin{lem}\label{lem:DE}
Suppose $(G,\vv\kk)$ and $(G^*,\vv\alpha)$ are dynamically equivalent. Then $(\mm A(G), \vv\kk)$ and $(\mm A(G^*), \vv\alpha)$ are dynamically equivalent. 
\end{lem}
\begin{proof}
    By definition of dynamical equivalence, the net reaction vector of $\yy_i$ in either $(G,\vv\kk)$ or $(G^*, \vv\alpha)$ is $\vv w_i$. Therefore by \Cref{prop:net-dir-vect} the net reaction vector of $\mm A(\yy_i)$ in either $(\mm A(G),\vv\kk)$ or $(\mm A(G^*), \vv\alpha)$ is $\mm M\vv w_i$. These net reaction vectors are equal, so $(\mm A(G), \vv\kk)$ and $(\mm A(G^*), \vv\alpha)$ are dynamically equivalent.
\end{proof}

We collect some observations about coordinate-wise exponentiation and logarithms of vectors, which will be used later for proving our main result.

\begin{lem}
\label{lem:exp-calc}
For $\xx \in \rrpp^n$, $\yy \in \rr^n$, and $\mm M \in \rr^{n \times n}$ whose columns are $\vv m_i$,  define the component-wise operations 
\eq{ 
    \exp(\yy) \coloneqq \begin{pmatrix} e^{y_1} \\  \vdots \\ e^{y_n} \end{pmatrix} , 
        \quad 
    \log(\xx) \coloneqq  \begin{pmatrix}  \log x_1 \\ \vdots \\  \log x_n \end{pmatrix} , 
         \quad \text{and} \qquad 
     \xx^{\mm M} \coloneqq \begin{pmatrix} \xx^{\vv m_1} \\ \vdots \\ \xx^{\vv m_n} \end{pmatrix}. 
}
Then we have 
\begin{enumerate}[label={(\roman*)}]
    \item $\log( \xx^{\vv v}) = \vv v^\top \log \xx$,
    \item $\log( \xx^{\mm M}) = \mm M^\top \log \xx$, 
    \item $\log( \xx^{\mm M\vv v}) = \vv v^\top  \log (\xx^\mm M)$, and
    \item $\xx^{\mm M_1\mm M_2} = (\xx^{\mm M_1})^{\mm M_2}$. 
\end{enumerate}       
\end{lem}

We leave the straightforward proof of \Cref{lem:exp-calc} to the reader.

\begin{prop}\label{lem:CB}
Suppose $(G^*, \vv\alpha)$ is complex-balanced. Then $(\mm A(G^*), \vv\alpha)$ is complex-balanced. 
\end{prop}
\begin{proof}
    By definition, $(G^*, \vv\alpha)$ being complex-balanced means there exists $\xx' \in \rrpp^n$ such that for every vertex $\yy_i \in V_s$, we have 
    \eqn{ \label{eq:flux-bal-Gcb}
        \sum_{j \neq i} \alpha_{ij}   &= \sum_{j \neq i} \alpha_{ji} (\xx')^{\yy_j - \yy_i} . 
    }
    We want to show that $(\mm A(G^*),\vv\alpha)$ is complex-balanced, i.e., we want to show that there exists a positive vector $\xx \in \rrpp^n$ such that for $i = 1,\ldots, |V_s|$, we have 
    \eqn{ \label{eq:flux-bal-AGcb}
        \sum_{j \neq i} \alpha_{ij}   &= \sum_{j \neq i} \alpha_{ji} \xx^{\mm A(\yy_j) - \mm A(\yy_i)} .
    }
    Given $\xx' \in \rrpp^n$ satisfying \eqref{eq:flux-bal-Gcb}, it suffices to show that there exists $\xx \in \rrpp^n$ such that for any $i$, $j$,  
    \eqn{\label{eq:binom-special} 
        (\xx')^{\yy_j-\yy_i} &= \xx^{\mm A(\yy_j) - \mm A(\yy_i)} ,
    }
    which is equivalent to $(\xx')^{\mm M^{-1}(\mm A(\yy_j) - \mm A(\yy_i))} = \xx^{\mm A(\yy_j) - \mm A(\yy_i)}$, or by \Cref{lem:exp-calc}, 
    \eq{ 
        (\mm A(\yy_j) - \mm A(\yy_i))^\top \log (\xx')^{\mm M^{-1}} = ( \mm A(\yy_j) - \mm A (\yy_i))^\top  \log \xx.
    }
    Choose $\xx = (\xx')^{\mm M^{-1}}$, which is a vector in $\rrpp^n$, defined independently of the vertex indices $i$, $j$. Then \eqref{eq:flux-bal-AGcb} follows from \eqref{eq:flux-bal-Gcb}.
\end{proof}

We now arrive at our main result. Our main motivation is to use this result to simplify the task of computing the disguised toric locus of a (more) complicated network. In \Cref{ex:doubletargets}, we demonstrate this with an affine image of the network in \Cref{fig:intro-mas}.

\begin{thm}\label{thm:main}
For any network $G$ in $\rr^n$ and any invertible affine transformation $\mm A$ of $\rr^n$,
\eq{
    \widehat{\mc K}(\mm A(G)) =  \widehat{\mc K}(G).
}
Specifically, if $(G,\vv\kk)$ is dynamically equivalent to a complex-balanced system $(G^*,\vv\alpha)$, then $(\mm A(G),\vv\kk)$ is dynamically equivalent to $(\mm A(G^*), \vv\alpha)$, which is complex-balanced.
\end{thm}

\begin{proof}
    Given a mass-action system $(G^*,\vv\alpha)$, applying \Cref{lem:DE} with $\mm A$ and $\mm A^{-1}$, we find that $(G,\vv\kk)$ is dynamically equivalent to  $(G^*,\vv\alpha)$ if and only if $(\mm A(G),\vv\kk)$ is dynamically equivalent to  $(\mm A(G^*),\vv\alpha)$. \Cref{lem:CB} shows $(G^*,\vv\alpha)$ is complex-balanced if and only if $(\mm A(G^*),\vv\alpha)$ is complex-balanced. As a result, $\widehat{\mc K}(\mm A(G)) =  \widehat{\mc K}(G)$. 
\end{proof}

Unlike the toric locus, computing the disguised toric locus is in general a difficult problem. Even for the relatively simple network shown in \Cref{fig:projective-counter-example-a}, the disguised toric locus is complicated:~as shown in~\cite[Section 4]{BCS22}, the answer breaks up into 4 cases; for 3 of these cases, the disguised toric locus is the whole parameter space and in the remaining case, it is defined by a quadratic inequality. The disguised toric locus here was computed using an algorithm based on real quantifier elimination~\cite{BCS22}.  

Let us see an example where  \Cref{thm:main} facilitates the computation of the disguised toric locus of one network, by considering an affinely equivalent network, whose disguised toric locus was already computed in~\cite[Section 6]{BCS22} and \cite[Example 5.2]{CraciunJinYu2019}.

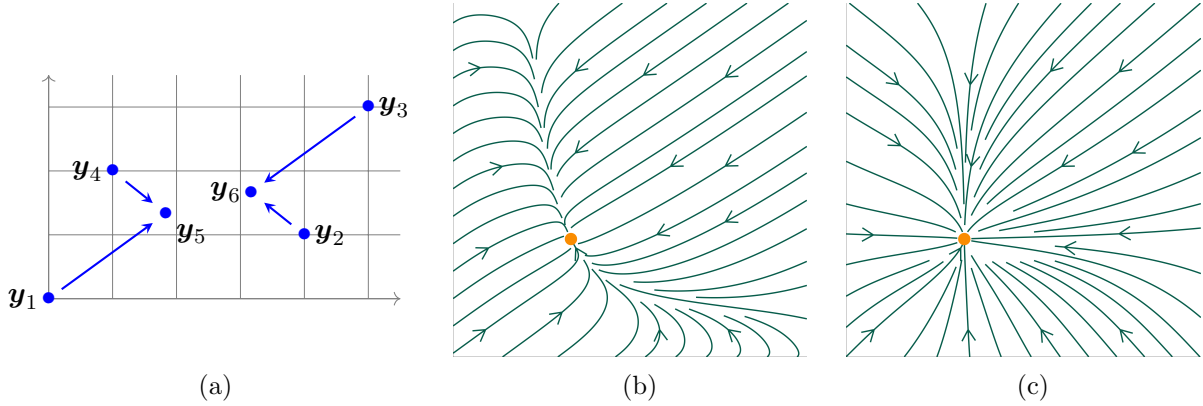
\begin{figure}[h!]
\centering
\hfill 
\begin{subfigure}[b]{0.35\textwidth}
    \begin{tikzpicture}[scale=0.85]
        \node at (0,-0.75) {};
        \draw [step=1, gray, very thin] (0,0) grid (5.5,3.5);
        \draw [ ->, gray] (0,0)--(5.5,0);
        \draw [ ->, gray] (0,0)--(0,3.5);
        
        \node [inner sep=2pt] (c1) at (11/6,4/3) {\blue{$\bullet$}};
        \node[inner sep=2pt]  (c2) at (3.167,5/3) {\blue{$\bullet$}};
        
        \node [inner sep=2pt] (1) at (0,0) {\blue{$\bullet$}};
        \node [inner sep=2pt](2) at (4,1) {\blue{$\bullet$}};
        \node [inner sep=2pt](3) at (5,3) {\blue{$\bullet$}};
        \node [inner sep=2pt](4) at (1,2) {\blue{$\bullet$}};
        
        \node at (0,0) [left] {$\yy_1$};
        \node at (1,2) [left] {$\yy_4$};
        \node at (4,1) [right] {$\yy_2$};
        \node at (5,3) [right] {$\yy_3$};
        \node at (11/6,4/3) [below right] {$\yy_5$};
        \node at (3.167,5/3) [left] {$\yy_6$};
        
        \draw [-{stealth}, thick, blue, transform canvas={xshift=0ex, yshift=0ex}] (1)--(c1) ;
        \draw [-{stealth}, thick, blue, transform canvas={xshift=-0ex, yshift=-0ex}] (4)--(c1) ;
                
        \draw [-{stealth}, thick, blue, transform canvas={xshift=0ex, yshift=0ex}] (2)--(c2) ;
        \draw [-{stealth}, thick, blue, transform canvas={xshift=-0ex, yshift=-0ex}] (3)--(c2) ;
	\end{tikzpicture}
	\caption{} 
	\label{fig:doubletargets-affine}
    \end{subfigure}
\hfill 
    \begin{subfigure}[b]{0.3\textwidth}
        \includegraphics[width=1.85in]{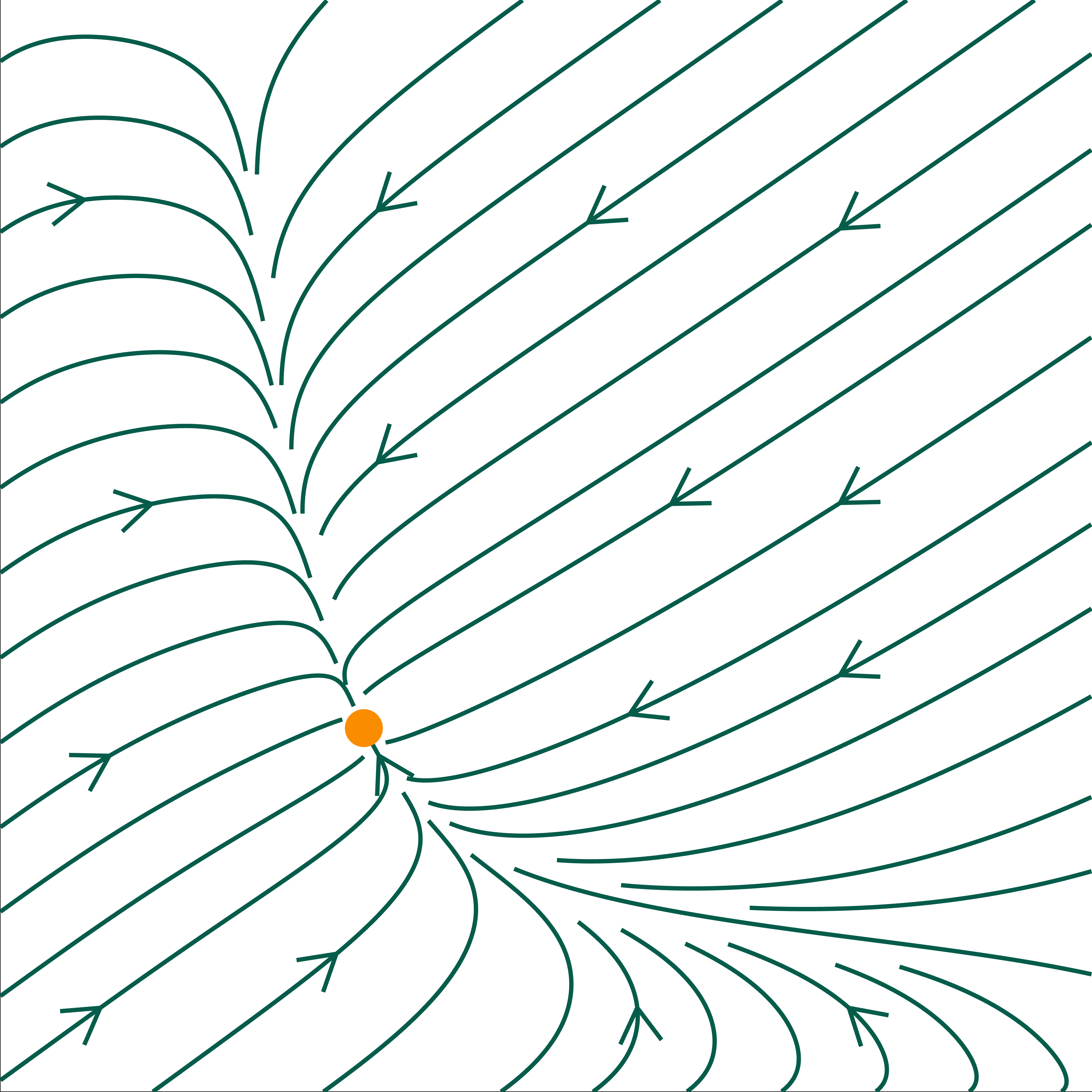}
    \caption{} 
	\label{fig:doubletargets-affine-trajectory}
    \end{subfigure}
\hfill 
    \begin{subfigure}[b]{0.3\textwidth}
        \includegraphics[width=1.85in]{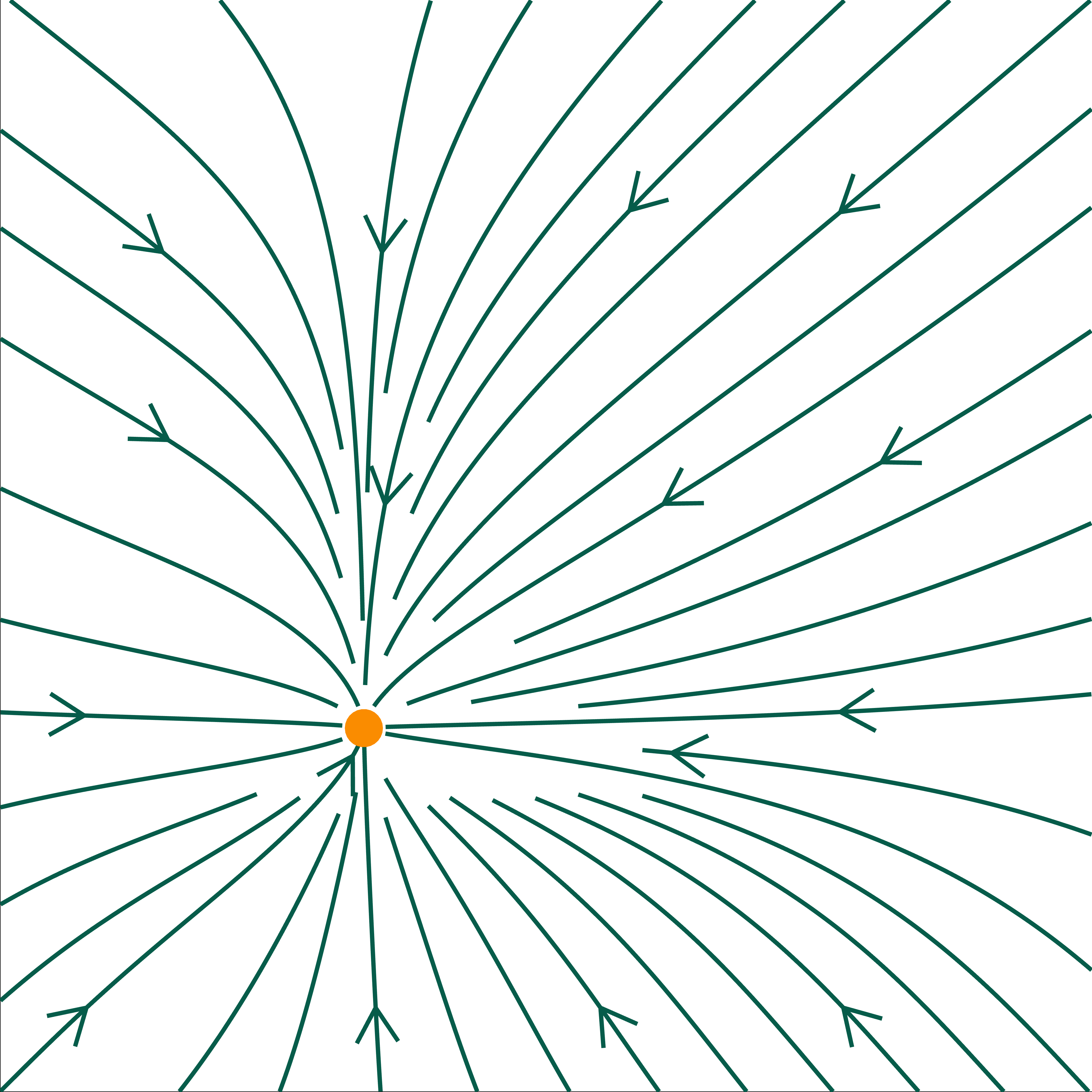}
    \caption{} 
	\label{fig:doubletargets-original-trajectory}
    \end{subfigure}
\hfill 
\caption{(a) The image of the network in \Cref{fig:intro-mas} under an invertible affine transformation (see \Cref{ex:doubletargets}), and (b) its phase portrait. For comparison, (c) the phase portrait of the network in \Cref{fig:intro-mas}. The rate constants are taken to be $\kk_1 = \cdots = \kk_4 = 1$ for simplicity.}
\label{fig:doubletargets-affine-tot}
\end{figure}

\begin{ex}
\label{ex:doubletargets}    
Consider the network $G$ in \Cref{fig:intro-mas}, whose image under the invertible linear transformation 
\eq{ 
    \mm M = \begin{pmatrix}
        4/3 & 1/2 \\ 
        1/3 & 1
    \end{pmatrix}
}
is shown in \Cref{fig:doubletargets-affine}. In other words, the network in \Cref{fig:doubletargets-affine} is affinely equivalent to the one in \Cref{fig:intro-mas}, whose disguised toric locus is 
\eqn{\label{eq:ex-Khat} 
    \widehat{\mc K}(G) = \left\{ \vv\kk \in \rrpp^4  \st \frac{1}{25} \leq \frac{\kk_1\kk_3}{\kk_2\kk_4} \leq 25 \right\} ,
}
where $\kk_i$ is the rate constant of the edge originating from $\vv y_i$, see~\cite[Section 6]{BCS22} and \cite[Example 5.2]{CraciunJinYu2019}. Note that the disguised toric locus is in general \emph{not} toric, even though it contains the toric locus. In this example, the toric locus is empty because $G$ is not weakly reversible. By \Cref{thm:main}, \eqref{eq:ex-Khat}  is also the disguised toric locus of the network $\mm M(G)$ in \Cref{fig:doubletargets-affine}. Indeed, we are not restricted to the network in \Cref{fig:doubletargets-affine}; instead the image of $G$ in \Cref{fig:intro-mas} under any invertible affine transformation (in any dimensions $n \geq 2$) would have the exact same disguised toric locus.

Consider the phase portraits of the dynamics generated by the network before (\Cref{fig:doubletargets-original-trajectory}) and after (\Cref{fig:doubletargets-affine-trajectory}) the affine transformation is applied, with $\kk_1 = \cdots = \kk_4 = 1$ for simplicity. Both mass-action systems are disguised toric, and for this choice of rate constants, both share the positive steady state $(1,1)^\top$. Both systems' dynamics are qualitatively the same, with the unique steady state being globally attracting~\cite{BorosHofbauer2020, Anderson2011}, since the systems are dynamically equivalent to complex-balanced systems with one strongly connected component.
\end{ex}

\begin{rmk}
A related though more restrictive notion of comparing mass-action systems was explored in  \cite{JohnstonSiegel2011}. The authors defined two mass-action systems to be \emph{linearly conjugate} if the trajectories of one are the images of the other under a linear map. When the linear map is the identity map, linear conjugacy captures the notion of dynamical equivalence. More generally, linear conjugacy is limited to scaling of individual axes and reordering coordinates~\cite[Lemma 3.1]{JohnstonSiegel2011}. Clearly, linearly conjugate mass-action systems are affinely equivalent, while the converse is not true. In \Cref{sec:discussion}, we see that affinely equivalent mass-action systems are generally \emph{not} topologically conjugate; in other words, it is not necessary for affinely equivalent mass-action systems to share the same qualitative dynamics (unless it is complex-balancing). 
\end{rmk}

\subsection{Extension to detailed-balanced systems}
\label{sec:DB-extension}

It is not difficult to show a parallel result where complex-balancing is replaced by \emph{detailed-balancing}. In particular, it suffices to show a result analogous to \Cref{lem:CB}. 
We first define detailed-balancing.

\begin{defn}
\label{def:DB} 
    A mass-action system $(G,\vv\kk)$ is \df{detailed-balanced} if there exists a positive steady state $\xx \in \rrpp^n$ such that for every reaction $\yy_i \to \yy_j \in E$, we have 
\begin{equation} \label{eq:DB}
    \kk_{ij} \xx^{\yy_i} = \kk_{ji} \xx^{\yy_j}.
\end{equation}
\end{defn}

A necessary condition for detailed-balancing is \emph{reversibility} of the network, i.e., if $\yy_i \to \yy_j \in E$, then $\yy_j \to \yy_i \in E$. Otherwise \eqref{eq:DB} cannot be satisfied.

\begin{thm}\label{thm:DB}
Suppose $(G,\vv\kk)$ is dynamically equivalent to a detailed-balanced system $(G^*, \vv\alpha)$. Let $\mm A$ be any invertible affine transformation. Then $(\mm A(G),\vv\kk)$ is dynamically equivalent to $(\mm A(G^*), \vv\alpha)$, which is detailed-balanced. 
\end{thm}
\begin{proof}
    By \Cref{prop:net-dir-vect}, the net reaction vector from $\mm A(\yy_i)$ in either $(\mm A(G), \vv\kk)$ or $(\mm A(G^*), \vv\alpha)$ is $\mm M \vv w_i$, hence by \Cref{lem:DE} $(\mm A(G),\vv\kk)$ and $(\mm A(G^*),\vv\alpha)$ are dynamically equivalent. It suffices to prove that $(\mm A(G^*), \vv\alpha)$ is detailed-balanced provided $(G^*,\vv\alpha)$ is detailed-balanced. 
    
    Suppose $\xx' \in \rrpp^n$ is a detailed-balanced steady state, so for any $i$, $j$ with $\alpha_{ij}$, $\alpha_{ji}  > 0$, we have $\alpha_{ij} (\xx')^{\yy_i} = \alpha_{ji} (\xx')^{\yy_j}$.
    As in the proof of Proposition \ref{lem:CB}, let $\xx = (\xx')^{\mm M^{-1}}$. By \Cref{lem:exp-calc},
    \eq{
        \xx^{\mm A(\yy_j)-\mm A(\yy_i)}=(\xx')^{\mm M^{-1}(\mm A(\yy_j)-\mm A(\yy_i))}=(\xx')^{\yy_j-\yy_i}.
    }
    In other words, $\alpha_{ij} \xx^{\mm A(\yy_i)} = \alpha_{ji} \xx^{\mm A(\yy_j)}$ for any $i \neq j$, so $\xx$ is a detailed-balanced steady state of $(\mm A(G^*),\vv\alpha)$.
\end{proof}

\section{Discussion}
\label{sec:discussion}

In this section, we explore two questions: how the dynamics of affinely equivalent mass-action systems are related, and whether \Cref{thm:main} can be extended beyond invertible affine transformations.  Towards the first question, we show that there is a canonical bijection between the sets of positive steady states of affinely equivalent mass-action systems. However, in general, local stability, capacity for multistationarity, and limit cycles are not preserved. Towards the latter question, at a minimum we require that the transformation preserve collinear vertices  per \Cref{rmk:collinear}. Invertible projective transformations both  generalize invertible affine maps and preserve collinearity. However, we provide examples of projective transformations that preserve neither dynamical equivalence nor complex-balancing.

We begin with some observations about the results from the previous section. 

First, invertible affine transformations preserve dynamical equivalence, see \Cref{lem:DE}. Recall that dynamical equivalence means that the network and rate constants may vary, while giving rise to the same ODE system. It is the system of ODEs that represents the keystone of our study, because it models the dynamical behaviour. Therefore, preservation of dynamical equivalence is an important asset of invertible affine transformations.

Second, we saw in \Cref{thm:main,thm:DB} that complex-balancing and  detailed-balancing are preserved under invertible affine transformations. The proofs of \Cref{thm:main,thm:DB} involve solving the equation
\eq{ 
    \mm \Delta \log \xx' = \mm \Delta \mm M ^\top \log \xx , 
}
where $\mm M$ is the derivative of $\mm A$ and $\mm \Delta$ is a matrix whose rows are $\yy_j - \yy_i$ with $i \neq j$ where the pairs $(i,j)$ are ordered in some manner, e.g., lexicographical order.

The set of complex-balancing steady states has a simple form. We let $\circ$ and $\exp({}\cdot{})$ denote component-wise multiplication and exponentiation respectively, i.e., 
\eq{ 
    \xx \circ \yy = (x_1y_1, x_2y_n, \ldots, x_ny_n)^\top 
    \quad \text{and} \quad 
    \exp(\xx) = (e^{x_1}, e^{x_2}, \ldots, e^{x_n})^\top. 
} 
If $S \subseteq \rr^n$, then $\exp(S) \coloneqq \{ \exp(\xx) \st  \xx \in S\}$. Then supposing that $\xx'$ is a steady state of a complex-balanced system $(G^*,\vv\alpha)$, the set of complex-balanced steady states for $(G^*,\vv\alpha)$ is precisely $\xx' \circ \exp(S^\perp)$~\cite{HornJackson1972, Feinberg1987, YuCraciun}, where $S$ is the stoichiometric subspace of $G^*$. In the proof of \Cref{lem:CB} we showed that $\xx = (\xx')^{\mm M^{-1}}$ is a complex-balanced steady state of $(\mm A(G^*),\vv\alpha)$, where  $(\xx')^{\mm M^{-1}}$ is a vector whose $i$th component is $\xx'$ exponentiated by the $i$th column of $\mm M^{-1}$, as introduced in \Cref{lem:exp-calc}. Put another way, if $\xx$ is a complex-balanced steady state for $(\mm A(G),\vv\alpha)$, then $\xx^{\mm M}$ is complex-balanced for $(G,\vv\kk)$, and $\xx^{\mm M} \in \xx' \circ \exp(S^\perp)$. 

We now show that there is a canonical bijection between the sets of positive steady states of arbitrary affinely equivalent mass-action systems, which may not be complex-balanced.

\begin{prop}
\label{prop:ss}
Let $\mm A$ be an invertible affine transformation, where $\mm A (\yy) = \mm M \yy + \vv b$. If $\xx$ is a positive steady state of $(\mm A(G),\vv\kk)$, then $\xx^{\mm M}$ is a positive steady state of  $(G,\vv\kk)$. The map $\xx \mapsto \xx^{\mm M}$ is a bijection between the set of positive steady states of $(\mm A(G),\vv\kk)$ and that of $(G,\vv\kk)$.
\end{prop}
\begin{proof}
Order the reactions from $1$ to $R$ with $\yy_{i_k}\to\yy_{j_k}$ the $k$th reaction and $\mm Y_s \in \rr^{n\times R}$ the matrix whose $k$th column is the source vertex $\yy_{i_k}$. Let $\mm D = \diag(\kk_1,\ldots, \kk_R)$, so that $\mm D  \xx^{\mm Y_s}$ is the vector whose components are the fluxes $\kk_{k} \xx^{\yy_{i_k}}$ of the reactions. Let $\mm \Gamma \in \rr^{n \times R}$ be the stoichiometric matrix, whose $k$th column is $\yy_{j_k}-\yy_{i_k}$. With this notation,  $(G,\vv\kk)$ is associated to the system of ODEs
\eqn{\label{eq:ss-original_mas} 
    \dot{\xx} &= \mm \Gamma \mm D  \xx^{\mm Y_s}, 
} 
while its image $\mm A(G)$ under $\mm A$ is associated to
\eqn{\label{eq:ss-affine_mas} 
    \dot{\xx} &= \mm M \mm \Gamma \mm D  \xx^{\mm M \mm Y_s + \vv b}
    = (\mm M \mm \Gamma \mm D  \xx^{\mm M \mm Y_s}) \cdot \xx^{\vv b}.
}
Because $\mm M$ is invertible and $\xx^{\vv b} \in \rrpp$ for any $\xx \in \rrpp^n$, we see that $\xx$ is a positive steady state of \eqref{eq:ss-affine_mas} if and only if $\vv 0 = \mm \Gamma \mm D \xx^{\mm M \mm Y_s}$. By \Cref{lem:exp-calc}, 
\eq{    
    \xx^{\mm M \mm Y_s} =  (\xx^{\mm M})^{\mm Y_s}, 
}
so $\xx^{\mm M}$ is a positive steady state of \eqref{eq:ss-original_mas}. The map $\xx \mapsto \xx^{\mm M}$ is a bijection between the sets of positive steady states of \eqref{eq:ss-original_mas} and \eqref{eq:ss-affine_mas}.
\end{proof}

\Cref{prop:ss} does \emph{not} imply that asymptotic stability or multistationarity is preserved under affine transformations, as the following example demonstrates. \Cref{ex:glycolysis} shows that limit cycles are also in general not preserved under affine transformations. 

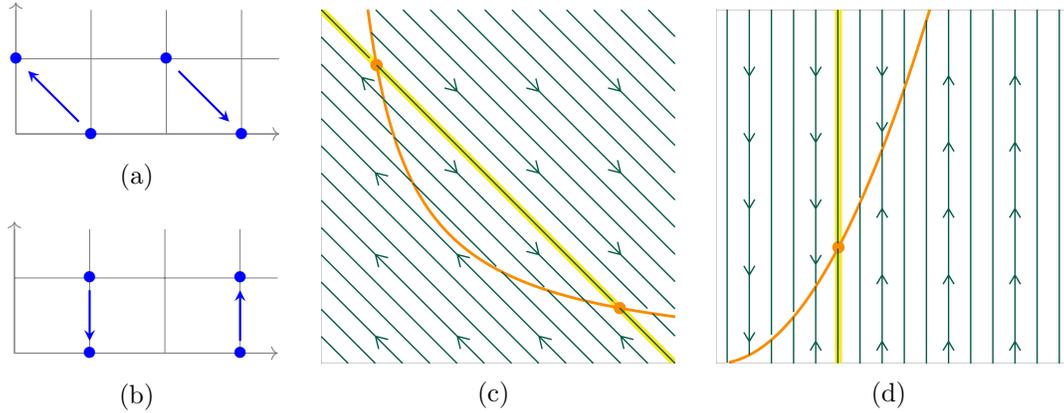
\begin{figure}[h!tb]
\centering\setlength\tabcolsep{0.15cm}
\begin{tabular}{ccc}
  \subcaptionbox{\label{fig:MS-network}}{
        \begin{tikzpicture}[scale=1]
        \draw [step=1, gray, very thin] (0,0) grid (3.5,1.65);
        \draw [ ->, gray] (0,0)--(3.5,0);
        \draw [ ->, gray] (0,0)--(0,1.75);
        
        \node [inner sep=2pt] (x) at (1,0) {\blue{$\bullet$}};
        \node [inner sep=2pt] (y) at (0,1) {\blue{$\bullet$}};
        \node [inner sep=2pt] (2xy) at (2,1) {\blue{$\bullet$}};
        \node [inner sep=2pt] (3x) at (3,0) {\blue{$\bullet$}};
        
        \draw [-{stealth}, thick, blue, transform canvas={xshift=0ex, yshift=0ex}] (x)--(y) ;
        \draw [-{stealth}, thick, blue, transform canvas={xshift=0ex, yshift=0ex}] (2xy)--(3x) ;
        \end{tikzpicture}
        }
    \\
    \subcaptionbox{\label{fig:MS-network-affine}}{
        \begin{tikzpicture}[scale=1]
        \node at (0,2.) {}; 
        \draw [step=1, gray, very thin] (0,0) grid (3.5,1.65);
        \draw [ ->, gray] (0,0)--(3.5,0);
        \draw [ ->, gray] (0,0)--(0,1.75);
        
        \node [inner sep=2pt] (x) at (1,1) {\blue{$\bullet$}};
        \node [inner sep=2pt] (y) at (1,0) {\blue{$\bullet$}};
        \node [inner sep=2pt] (2xy) at (3,0) {\blue{$\bullet$}};
        \node [inner sep=2pt] (3x) at (3,1) {\blue{$\bullet$}};
        
        \draw [-{stealth}, thick, blue, transform canvas={xshift=0ex, yshift=0ex}] (x)--(y) ;
        \draw [-{stealth}, thick, blue, transform canvas={xshift=0ex, yshift=0ex}] (2xy)--(3x) ;
        \end{tikzpicture}
        }    
    &\multirow[b]{-1.8}{*}[1.52in]{\subcaptionbox{\label{fig:MS-traj}}{
        \includegraphics[width=1.85in]{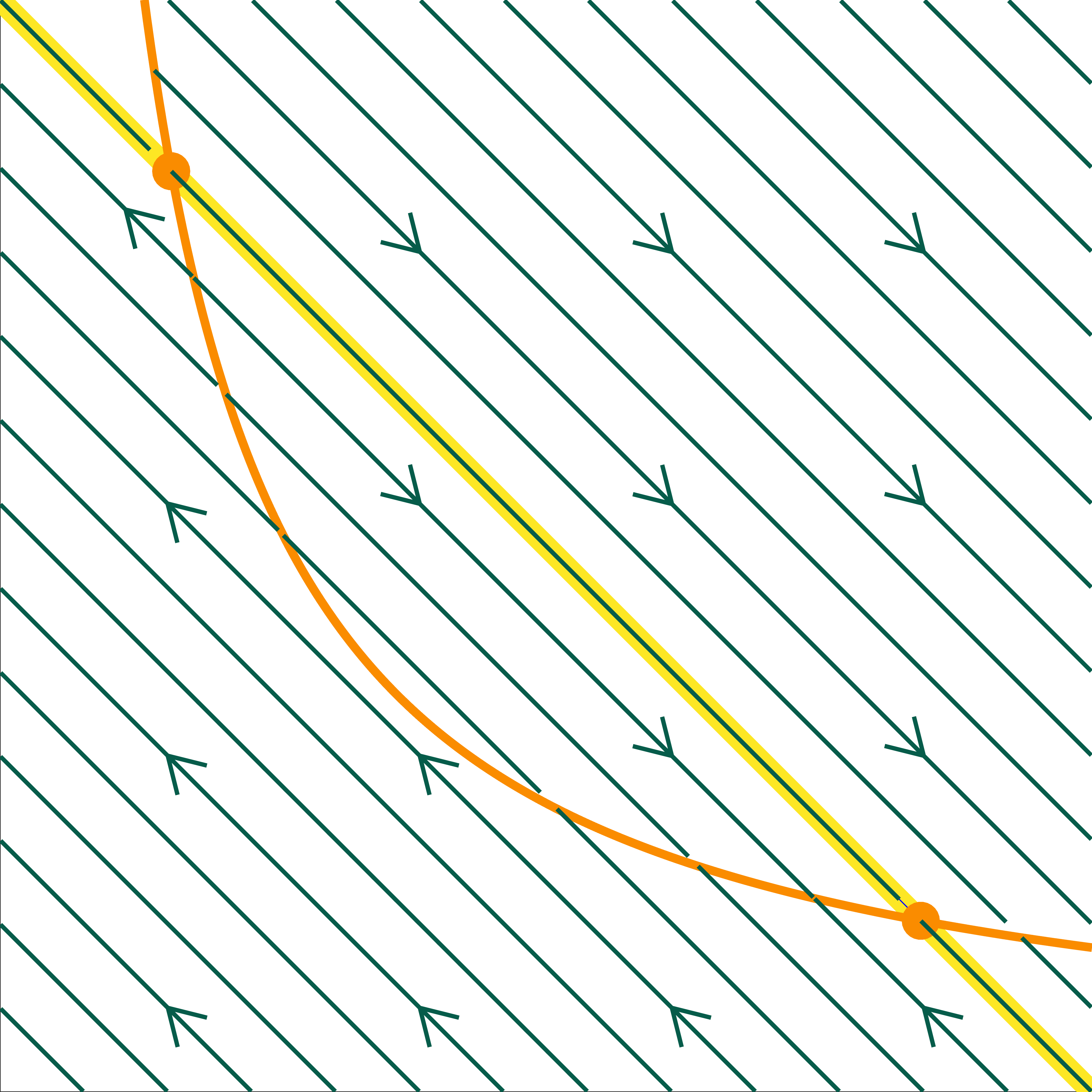}
    }}
    & \multirow[b]{-1.8}{*}[1.52in]{\subcaptionbox{\label{fig:MS-aff-traj}}{
        \includegraphics[width=1.85in]{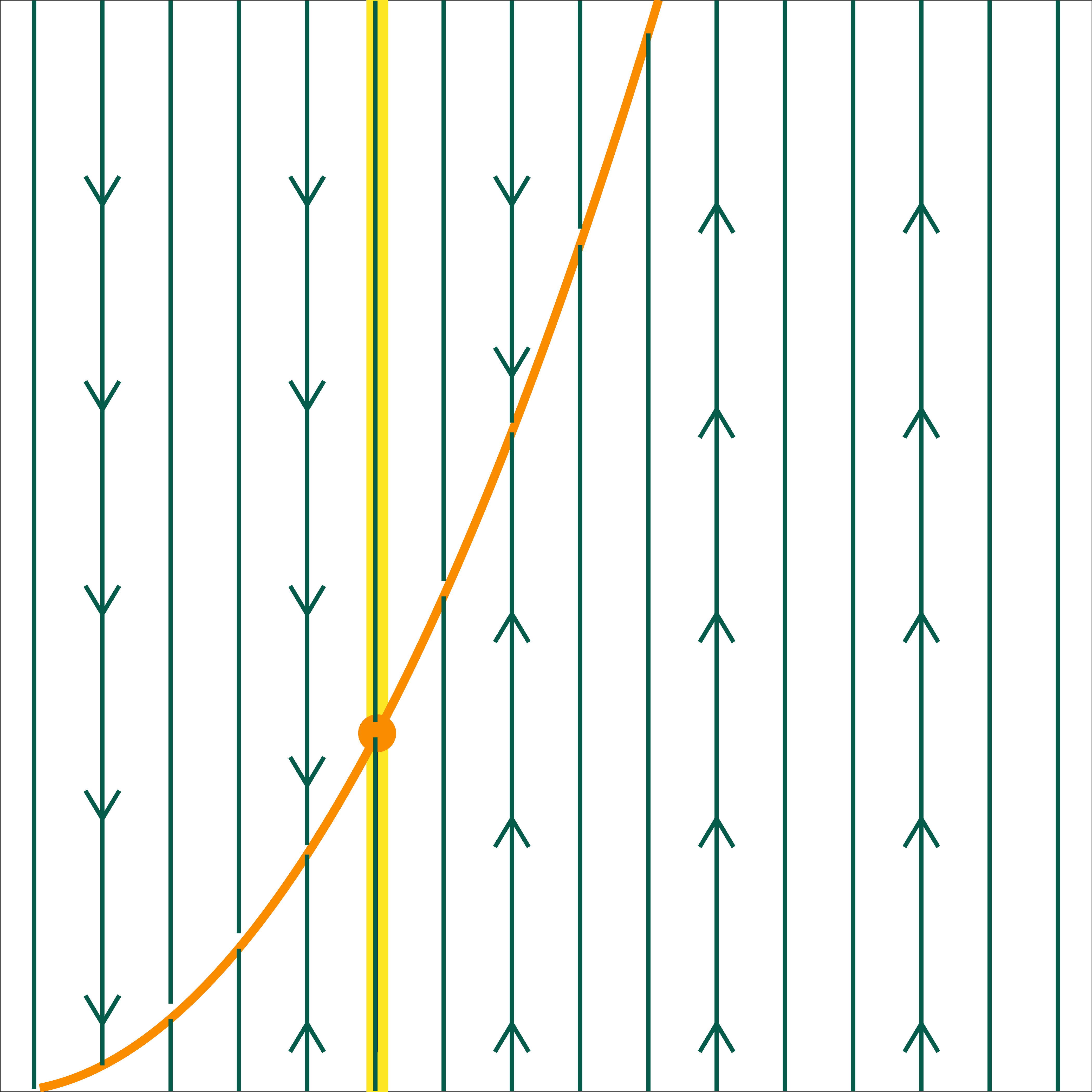}
    }} 
\end{tabular}
\caption{(a) A multistationary mass-action system ($\kk_1 = \kk_2 = 1$) and (c) its phase portrait. (b) Its image under an affine map as in \Cref{ex:MS-not}, and (d) its phase portrait, showing the system is \emph{not} multistationary. In (c) and (d), particular stoichiometric compatibility classes are highlighted in yellow.}
\label{fig:MS}
\end{figure}

\begin{ex}[\textbf{Affine transformations do not preserve multistationarity}] 
\label{ex:MS-not}
    Consider the mass-action system shown in \Cref{fig:MS-network}, with rate constants $\kk_1 = \kk_2 = 1$. Its phase portrait, shown in \Cref{fig:MS-traj}, shows that the system is \emph{multistationary}, i.e., there exists a stoichiometric compatibility class with more than one positive steady state. The figure highlights a particular stoichiometric compatibility class (in yellow) with two steady states. 

    We map the mass-action system under the affine transformation 
    \eq{ 
        \mm A (\yy) = \begin{pmatrix} 1 & 1 \\ 0 & -1 \end{pmatrix} \yy + \begin{pmatrix} 0 \\ 1 \end{pmatrix}, 
    }
    resulting in the network in \Cref{fig:MS-network-affine}. This mass-action system is \emph{not} multistationary. Indeed, this latter network is \emph{not capable} of multistationarity, because for generic $\kk_1$, $\kk_2$, the steady state set is given by 
    \eq{ 
        y^* = \frac{\kk_2}{\kk_1} (x^*)^2, 
    }
    where the value of $x^* > 0$ defines the stoichiometric compatibility class. Note that this example does not contradict \Cref{prop:ss}, which asserts bijection on the \emph{sets} of positive steady states, not how the sets intersect with any stoichiometric compatibility class. The latter is related to multistationarity.
\end{ex}

\begin{figure}[h!]
\centering
    \begin{subfigure}[b]{0.48\textwidth}\centering 
    \begin{tikzpicture}[scale=1]
        \draw [step=1, gray!50!white, thin] (-1.25,0) grid (3.5,3.25);
        \node at (0,2.75) {};
		\node at (0,-0.25) {};
            \draw [->, gray] (-1.25,0)--(3.5,0);
            \draw [->, gray] (0,0)--(0,3.25);
            \node [inner sep=0.2pt, outer sep=0pt] (1) at (0,0) {\blue{$\bullet$}};
            \node [inner sep=0.2pt, outer sep=0pt] (2) at (0,1) {\blue{$\bullet$}};
            \node [inner sep=0.2pt, outer sep=0pt] (3) at (1,0) {\blue{$\bullet$}};
            \node [inner sep=0.2pt, outer sep=0pt] (4) at (2,1) {\blue{$\bullet$}};
            \node [inner sep=0.2pt, outer sep=0pt] (5) at (3,0) {\blue{$\bullet$}};
        
            \node [below of=1, node distance=9pt] {$\yy_1$};
            \node [above of=2, node distance=9pt] {$\yy_2$};
            \node [below of=3, node distance=9pt] {$\yy_3$};
            \node [above of=4, node distance=9pt] {$\yy_4$};
            \node [below of=5, node distance=9pt] {$\yy_5$};

            \draw [-{stealth}, thick, blue, transform canvas={xshift=-0.35ex, yshift=0ex}] (1)--(2) ;
            \draw [-{stealth}, thick, blue, transform canvas={xshift=0.35ex, yshift=-0ex}] (2)--(1) ;
            \draw [-{stealth}, thick, blue, transform canvas={yshift=0.35ex, xshift=0ex}] (1)--(3) ;
            \draw [-{stealth}, thick, blue, transform canvas={yshift=-0.35ex, xshift=-0ex}] (3)--(1) ;
            \draw [-{stealth}, thick, blue, transform canvas={xshift=0.2ex, yshift=0.2ex}] (2)--(3) ;
            \draw [-{stealth}, thick, blue, transform canvas={xshift=-0.2ex, yshift=-0.2ex}] (3)--(2) ;
            \draw [-{stealth}, thick, blue, transform canvas={xshift=0.2ex, yshift=0.2ex}] (4)--(5) ;
            \draw [-{stealth}, thick, blue, transform canvas={xshift=-0.2ex, yshift=-0.2ex}] (5)--(4) ;
		\end{tikzpicture}
		\caption{}
        \label{fig:glycolysis-network}
    \end{subfigure}
\hspace{0.1cm} 
   \begin{subfigure}[b]{0.48\textwidth}\centering 
    \begin{tikzpicture}[scale=1]
        \draw [step=1, gray!50!white, thin] (-1.25,0) grid (3.5,3.25);
        \node at (0,2.75) {};
		\node at (0,-0.25) {};
            \draw [->, gray] (-1.25,0)--(3.5,0);
            \draw [->, gray] (0,0)--(0,3.25);
            \node [inner sep=0.2pt, outer sep=0pt] (1) at (0,0) {\blue{$\bullet$}};
            \node [inner sep=0.2pt, outer sep=0pt] (2) at (-0.866,0.5) {\blue{$\bullet$}};
            \node [inner sep=0.2pt, outer sep=0pt] (3) at (0.5,0.866) {\blue{$\bullet$}};
            \node [inner sep=0.2pt, outer sep=0pt] (4) at (0.134,2.2321) {\blue{$\bullet$}};
            \node [inner sep=0.2pt, outer sep=0pt] (5) at (1.5,2.5981) {\blue{$\bullet$}};

            \node [below of=1, node distance=11pt] {$\yy'_1$};
            \node [above of=2, node distance=11pt] {$\yy'_2$};
            \node [above right of=3, node distance=11pt] {$\yy'_3$};
            \node [left of=4, node distance=11pt] {$\yy'_4$};
            \node [right of=5, node distance=11pt] {$\yy'_5$};
            
            \draw [-{stealth}, thick, blue, transform canvas={xshift=-0.2ex, yshift=-0.2ex}] (1)--(2) ;
            \draw [-{stealth}, thick, blue, transform canvas={xshift=0.2ex, yshift=0.2ex}] (2)--(1) ;
            \draw [-{stealth}, thick, blue, transform canvas={yshift=0.2ex, xshift=-0.2ex}] (1)--(3) ;
            \draw [-{stealth}, thick, blue, transform canvas={yshift=-0.2ex, xshift=0.2ex}] (3)--(1) ;
            \draw [-{stealth}, thick, blue, transform canvas={xshift=0ex, yshift=0.3ex}] (2)--(3) ;
            \draw [-{stealth}, thick, blue, transform canvas={xshift=-0ex, yshift=-0.3ex}] (3)--(2) ;
            \draw [-{stealth}, thick, blue, transform canvas={xshift=0ex, yshift=0.3ex}] (4)--(5) ;
            \draw [-{stealth}, thick, blue, transform canvas={xshift=-0ex, yshift=-0.3ex}] (5)--(4) ;
		\end{tikzpicture}
		\caption{}
        \label{fig:glycolysis-network-affine}
    \end{subfigure}
    
\vspace{0.5cm}
    \begin{subfigure}[b]{0.3\textwidth}
    \centering 
    \includegraphics[width=1.85in]{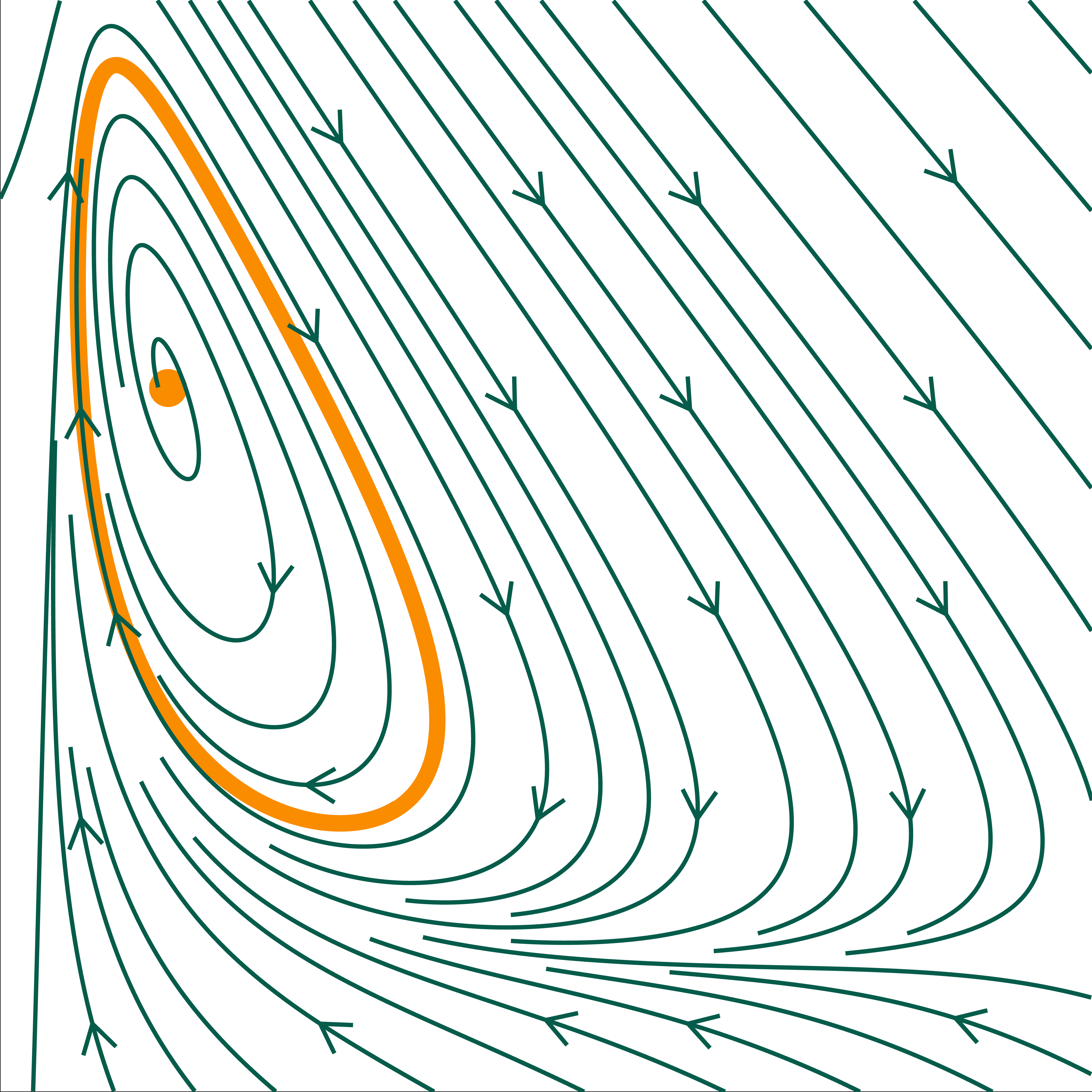}
    \caption{} 
	\label{fig:glycolysis-traj}
    \end{subfigure}
\hspace{3.05cm}
    \begin{subfigure}[b]{0.3\textwidth}
    \centering 
    \includegraphics[width=1.85in]{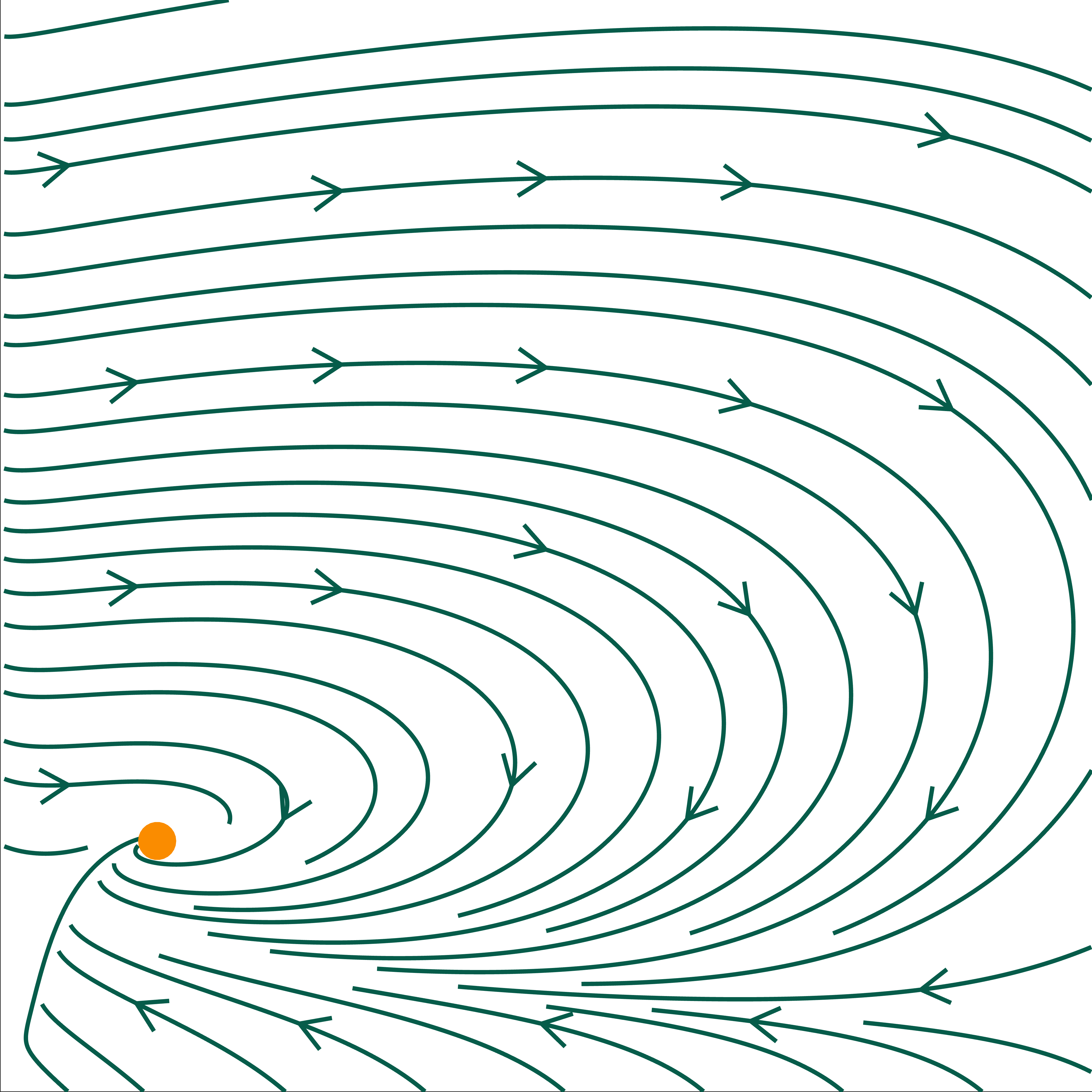}
    \caption{} 
	\label{fig:glycolysis-traj-affine}
    \end{subfigure}
\caption{(a) A version of the Brusselator and (b) its image under rotation by $\frac{\pi}{3}$. Their phase portraits are shown in (c) and (d) respectively. Note that the limit cycles are not in general preserved under invertible affine transformations. See \Cref{ex:glycolysis} for the rate constants used to generate the phase portraits. }
\label{fig:glycolysis}
\end{figure}
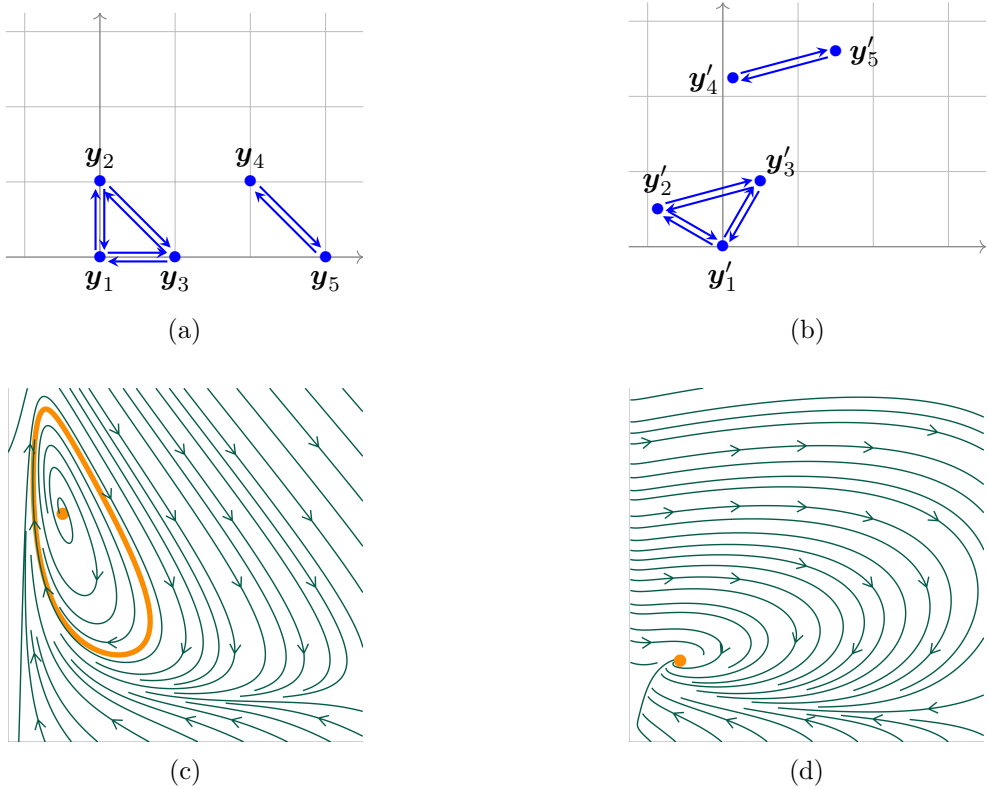

\begin{ex}[\textbf{Affine transformations do not preserve limit cycles}] 
\label{ex:glycolysis} 
    Let us consider in \Cref{fig:glycolysis-network}  a version of the  Brusselator~\cite{prigogine1979irreversibility}. A version of this is one of the prime examples of a Hopf bifurcation. For certain choices of rate constants, the mass-action system admits a stable limit cycle, as shown in \Cref{fig:glycolysis-traj} where we take 
    \eq{ 
        \kk_{12} = 0.5, \quad 
        \kk_{21} = \kk_{13} = \kk_{54} = 0.1, \quad 
        \kk_{23} = \kk_{32} = 0.01, \quad 
        \kk_{31} = \kk_{45} = 1. 
    }
    It is not difficult to show  that the stable limit cycle can disappear when we consider an affinely equivalent mass-action system. For example, if the network is rotated by $\frac{\pi}{3}$, as shown in \Cref{fig:glycolysis-network-affine}, there is no limit cycle, as seen in its phase portrait in \Cref{fig:glycolysis-traj-affine}. 
    Moreover, what was an unstable steady state is now stable under the affine transformation. This demonstrates that invertible affine transformations do not in general preserve qualitative dynamics.
\end{ex}

Although invertible affine transformations do not preserve multistationarity or limit cycles, they do preserve many important classes of reaction networks, especially those with dynamical implications. For example, since network structure is preserved, the images of reversible and weakly reversible networks are reversible and weakly reversible respectively. For such mass-action systems, there exists a steady state in every stoichiometric compatibility class~\cite{Boros2019}. Moreover, these systems are conjectured to be \emph{persistent} (no trajectory with positive initial condition has an $\omega$-limit point on the boundary of $\rrpp^n$) and \emph{permanent} (any trajectory eventually converges to a compact subset of the stoichiometric compatibility class)~\cite{Feinberg1987, CraciunNazarovPantea2013}.  

Two more general classes of reaction networks are also preserved under invertible affine transformations:  \emph{endotactic} networks~\cite{CraciunNazarovPantea2013} and \emph{strongly endotactic} networks~\cite{GopalkrishnanMillerShiu2014}. Geometrically, a network is endotactic if none of the reaction vectors point outside of the convex hull of source vertices, or the \emph{Newton polytope} of the network. A network is strongly endotactic if it is endotactic and on every facet of the Newton polytope, there is a reaction vector pointing away from the facet. Since affine transformations preserve convexity and half-spaces in $\rr^n$, endotacticity and strong endotacticty are preserved. 

In terms of dynamics, strongly endotactic mass-action systems are permanent~\cite{GopalkrishnanMillerShiu2014}. This was utilized to show that the Global Attractor Conjecture holds for complex-balanced systems with only one connected component. The slightly more general endotactic mass-action systems are conjectured to be permanent, where the case of $\rr^2$ was proved in \cite{CraciunNazarovPantea2013}. Again, the permanence of endotactic networks, which include all weakly reversible networks, is used towards solving the Global Attractor Conjecture in general~\cite{Craciun2019, CraciunNazarovPantea2013, CraciunDeshpande2020}.

\subsection{The disguised toric locus and projective transformations}
\label{sec:projective}

Invertible affine transformations were natural in our study of disguised toricity, since for the purpose of dynamical equivalence collinear vertices must be preserved.  Projective transformations also preserve collinear points. As such, we consider networks that are mapped via invertible projective maps. The examples presented in this section demonstrate that such maps do not preserve disguised toricity. Therefore, the class of invertible affine transformations is most naturally associated to disguised toric systems. In the remainder of this section, we give a series of examples showing that invertible projective transformations do not preserve the disguised toric locus (\Cref{ex:projective-counter-example}), do not preserve complex-balancing (\Cref{ex:CB-not_preserved-proj}), and do not preserve dynamical equivalence (\Cref{ex:projective-not-->DE}). 

\begin{defn}
\label{def:invertible-proj}
An \df{invertible projective transformation} on $\rr^n$ is a function of the form
\eq{
    P(\yy)=\left(\frac{\ell_1(\yy)}{\ell_0(\yy)},\dots,\frac{\ell_n(\yy)}{\ell_0(\yy)}\right)^\top,
}
where $\yy\in\rr^n$, $\ell_i(\yy)=a_{i0}+\sum_{j=1}^n a_{ij}y_j$ for $a_{ij}\in\rr$, and the $(n+1)\times(n+1)$ matrix $(a_{ij})_{0\leq i,j\leq n}$ is invertible over $\rr$.
\end{defn}

\begin{defn}
\label{def:homeomorphism} 
Let $G = (V,E)$ be a network in $\rr^n$. Let $P=(\frac{\ell_1}{\ell_0},\dots,\frac{\ell_n}{\ell_0})$ be an invertible projective transformation on $\rr^n$ such that $V\subset\rr^n$ does not intersect the locus where $\ell_0=0$. Let $P(V) \coloneqq \{ P(\yy) \st \yy \in V\}$, and $P(E) \coloneqq \{ P(\yy_i) \to P(\yy_j) \st \yy_i \to \yy_j \in E\}$. The \df{image of $G$ under $P$} is the graph $P(G) = (P(V), P(E))$.
\end{defn}

As before, if $(G,\vv\kk)$ is a mass-action system, then its image under $P$ is $(P(G),\vv\kk)$, where $\kk_{ij}$ is the rate constant of $P(\yy_i) \to P(\yy_j)$ for any $\yy_i \to \yy_j \in E$. 

\begin{figure}[h!]
    \centering
    \begin{subfigure}[b]{0.45\textwidth}\centering 
    \begin{tikzpicture}[scale=1.35]
        \draw [step=1, gray!50!white, thin] (0,0) grid (3.5,3.5);
        \node at (0,2.75) {};
		\node at (0,-0.25) {};
            \draw [->, gray] (0,0)--(3.5,0);
            \draw [->, gray] (0,0)--(0,3.5);
            \node [inner sep=0pt, outer sep=0pt] (1) at (3,0) {\blue{$\bullet$}};
            \node [inner sep=0pt, outer sep=0pt] (2) at (2,1) {\blue{$\bullet$}};
            \node [inner sep=0pt, outer sep=0pt] (3) at (1,2) {\blue{$\bullet$}};
            \node [inner sep=0pt, outer sep=0pt] (4) at (0,3) {\blue{$\bullet$}};
            \node [outer sep=1pt] at (1) [above right] {$\yy_1$};
            \node [outer sep=1pt] at (2) [above right] {$\yy_2$};
            \node [outer sep=1pt] at (3) [below left] {$\yy_3$};
            \node [outer sep=1pt] at (4) [left] {$\yy_4$};

            \draw [-{stealth}, thick, blue, transform canvas={xshift=0.2ex, yshift=0.2ex}] (1)--(2) ;
            \draw [-{stealth}, thick, blue, transform canvas={xshift=-0.2ex, yshift=-0.2ex}] (2)--(1) ;           
            \draw [-{stealth}, thick, blue, transform canvas={xshift=0.2ex, yshift=0.2ex}] (2)--(3) ;
            \draw [-{stealth}, thick, blue, transform canvas={xshift=-0.2ex, yshift=-0.2ex}] (3)--(2) ;   
            
            \draw [-{stealth}, thick, blue, transform canvas={xshift=0.2ex, yshift=0.2ex}] (3)--(4) ;
            \draw [-{stealth}, thick, blue, transform canvas={xshift=-0.2ex, yshift=-0.2ex}] (4)--(3) ;   
            
            \draw [-{stealth}, thick, blue, transform canvas={xshift=0.6ex, yshift=0.6ex}] (1)--(4) ;
            \draw [-{stealth}, thick, blue, transform canvas={xshift=-0.6ex, yshift=-0.6ex}] (4)--(1) ;
		\end{tikzpicture}
		\caption{}
        \label{fig:projective-counter-example-a}
    \end{subfigure}
    \begin{subfigure}[b]{0.45\textwidth}\centering 
    \begin{tikzpicture}[scale=1.35]
        \draw [step=1, gray!50!white, thin] (0,0) grid (3.5,3.5);
        \node at (0,2.75) {};
		\node at (0,-0.25) {};
            \draw [->, gray] (0,0)--(3.5,0);
            \draw [->, gray] (0,0)--(0,3.5);
            \node [inner sep=0pt, outer sep=0pt] (1) at (1,2) {\blue{$\bullet$}};
            \node [inner sep=0pt, outer sep=0pt] (2) at (0.75,2.25) {\blue{$\bullet$}}; 
            \node [inner sep=0pt, outer sep=0pt] (3) at (0,3) {\blue{$\bullet$}};
            \node [inner sep=0pt, outer sep=0pt] (4) at (2,1) {\blue{$\bullet$}};
            \node [outer sep=0pt] at (1) [ right] {\,\,\,\,$\yy'_1$};
            \node [outer sep=0pt] at (2) [ left] {$\yy'_2$\,\,\,\,};
            \node [outer sep=1pt] at (3) [ left] {$\yy'_3$};
            \node [outer sep=0pt] at (4) [above right] {$\yy'_4$};

            \draw [-{stealth}, thick, blue, transform canvas={xshift=0.2ex, yshift=0.2ex}] (4)--(1) ;
            \draw [-{stealth}, thick, blue, transform canvas={xshift=-0.2ex, yshift=-0.2ex}] (1)--(4) ;      
            \draw [-{stealth}, thick, blue, transform canvas={xshift=-0.2ex, yshift=-0.2ex}] (2)--(1) ;   
            \draw [-{stealth}, thick, blue, transform canvas={xshift=0.2ex, yshift=0.2ex}] (1)--(2) ;   
            
            \draw [-{stealth}, thick, blue, transform canvas={xshift=0.2ex, yshift=0.2ex}] (2)--(3) ;
            \draw [-{stealth}, thick, blue, transform canvas={xshift=-0.2ex, yshift=-0.2ex}] (3)--(2) ;   
            
            \draw [-{stealth}, thick, blue, transform canvas={xshift=0.6ex, yshift=0.6ex}] (4)--(3) ;
            \draw [-{stealth}, thick, blue, transform canvas={xshift=-0.6ex, yshift=-0.6ex}] (3)--(4) ;
		\end{tikzpicture}
		\caption{}
		\label{fig:projective-counter-example-b}
    \end{subfigure}
\caption{(a) A complete graph $G$ on four points on a line. (b) The image $P(G)$ under the projective map \eqref{eq:ex-proj}. \Cref{ex:projective-counter-example} shows that $\widehat{\mc K}(P(G)) \setminus \widehat{\mc K}(G) \neq \varnothing$, which implies that the property of being disguised toric is not preserved by projective maps in general.}
\label{fig:projective-counter-example}
\end{figure}
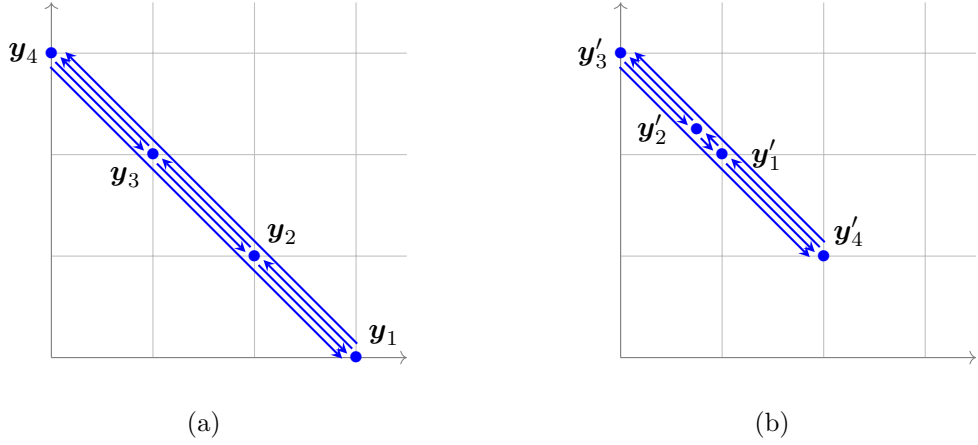 

\begin{ex}[\textbf{Projective transformations do not preserve the property of being disguised toric}] 
\label{ex:projective-counter-example}
    Let $G$ be the complete graph on the vertices $(3,0)$, $(2,1)$, $(1,2)$, $(0,3)$, as shown in \Cref{fig:projective-counter-example-a}. The disguised toric locus of this graph was characterized in \cite{BCS22}; we follow the authors' notation (except we have written $\kk_{ij}$ in place of $k_{ij}$ for the rate constant of $\yy_i \to \yy_j$). In  \cite[Section 4]{BCS22} it was shown that for the network in \Cref{fig:projective-counter-example-a}, also called \emph{the quadrilateral on a line}, the positive orthant of the parameters $\vv\kk$ is partitioned into four regions (see \cite[Proposition 4.1]{BCS22}),  based on linear inequalities on $\kk_{ij}$. Three of these regions always give rise to disguised toric systems (these regions are called the  single-sign-change chambers in \cite[Definition 4.2]{BCS22}), whereas for the last region the parameters need to satisfy an additional nonlinear polynomial  inequality for the system to be disguised toric \cite[Theorem 4.3]{BCS22}. 

    Let us now consider the following projective transformation 
    \eqn{\label{eq:ex-proj}
        P(x,y)=\left(\frac{-17x+7y+3}{-17x+3y+3},\frac{-33x+3y+3}{-17x+3y+3}\right) .
    }
    The vertices are mapped to 
    \eq{ 
        \yy'_4 = P(0,3)=(2,1),  \qquad &  \yy'_2 = P(2,1)= (6/7, 15/7), \\
        \yy'_3 = P(1,2)=(0,3) ,  \qquad & \yy'_1 = P(3,0)=(1,2). 
    }
    So $(P(G),\vv\kk)$, consisting of the reactions 
    \begin{center} 
    \begin{tikzpicture}
        \node (1a) at (0,0) [left] {$(0,3)$};
        \node (1b) at (1.5,0) [right] {$(\frac{6}{7}, \frac{15}{7})$}; 
        \draw [revrxn, transform canvas={yshift=1.5pt}] (1a)--(1b) node [midway, above] {\footnotesize $\kk_{32}$};
        \draw [revrxn, transform canvas={yshift=-1.5pt}] (1b)--(1a) node [midway, below] {\footnotesize $\kk_{23}$};
        \begin{scope}[shift={(5,0)}] 
        \node (1a) at (0,0) [left] {$(\frac{6}{7}, \frac{15}{7})$};
        \node (1b) at (1.5,0) [right] {$(1,2)$}; 
        \draw [revrxn, transform canvas={yshift=1.5pt}] (1a)--(1b) node [midway, above] {\footnotesize $\kk_{21}$};
        \draw [revrxn, transform canvas={yshift=-1.5pt}] (1b)--(1a) node [midway, below] {\footnotesize $\kk_{12}$};
        \end{scope}
        \begin{scope}[shift={(0,-1)}] 
        \node (1a) at (0,0) [left] {$(1,2)$};
        \node (1b) at (1.5,0) [right] {$(2,1)$}; 
        \draw [revrxn, transform canvas={yshift=1.5pt}] (1a)--(1b) node [midway, above] {\footnotesize $\kk_{14}$};
        \draw [revrxn, transform canvas={yshift=-1.5pt}] (1b)--(1a) node [midway, below] {\footnotesize $\kk_{41}$};
        \end{scope}
        \begin{scope}[shift={(5,-1)}] 
        \node (1a) at (0,0) [left] {$(0,3)$};
        \node (1b) at (1.5,0) [right] {$(2,1)$}; 
        \draw [revrxn, transform canvas={yshift=1.5pt}] (1a)--(1b) node [midway, above] {\footnotesize $\kk_{34}$};
        \draw [revrxn, transform canvas={yshift=-1.5pt}] (1b)--(1a) node [midway, below] {\footnotesize $\kk_{43}$};
        \end{scope}
    \end{tikzpicture}
    \end{center}
    is also a complete graph with four vertices, as shown in \Cref{fig:projective-counter-example-b}. 

    Now, consider a vector of positive real numbers $\vv\kk = (\kk_{ij})_{ij}$ with the following properties:
    \begin{enumerate}[label={(\arabic*)}]
    \item\label{proj-counter-ex::chamber4}  $\kk_{21} - \kk_{23} -  2\kk_{24} > 0 $   and     $2\kk_{31} + \kk_{32} - \kk_{34}  < 0$,   
    
    \item\label{proj-counter-ex::chamber4-CB}  $ (\kk_{12} + 2 \kk_{13} + 3 \kk_{14})  (3 \kk_{41} + 2 \kk_{42} + \kk_{43})  < \vert \kk_{21} - \kk_{23} - 2 \kk_{24}\vert  \vert 2 \kk_{31} + \kk_{32} - \kk_{34} \vert$,  
    
    \item\label{proj-counter-ex::chamberneq4} 
    $6\kk_{23} < \kk_{21} + 8\kk_{24} $ or 
    $\kk_{14} < \frac{1}{7} \kk_{12} + \kk_{13}$. 
    \end{enumerate}
    For example, we can take
    \begin{gather*}
        \kk_{21} = 4,\quad  \kk_{23} = 1,\quad  \kk_{24} = 1,\quad  \kk_{31} = \frac{1}{8},\quad  \kk_{32} = \frac{1}{2},\quad  \kk_{34} = 1,\quad  \\
        \kk_{12} = \frac{1}{512},\quad \kk_{13} = \frac{1}{128},\quad  \kk_{14} = \frac{1}{128},\quad  \kk_{41} = 1,\quad  \kk_{42} = 1,\quad  \kk_{43} = 1.  
    \end{gather*} 

    We claim that any choice of $\vv\kk$ satisfying conditions \ref{proj-counter-ex::chamber4}--\ref{proj-counter-ex::chamberneq4} implies that $(G,\vv\kk)$ is \emph{not} a disguised toric dynamical system, but $(P(G), \vv\kk)$ is. More precisely, condition \ref{proj-counter-ex::chamber4} says that $(G,\vv\kk)$ lives in Chamber $4$~\cite[Theorem 4.3(1)]{BCS22}, with the net reaction vector $\vv w_3$ pointing towards $\vv y_4$ and the net reaction vector $\vv w_2$ pointing towards $\vv y_1$. In this case, $(G,\vv\kk)$ is \emph{not} disguised toric exactly when condition \ref{proj-counter-ex::chamber4-CB} holds. In other words, conditions \ref{proj-counter-ex::chamber4}--\ref{proj-counter-ex::chamber4-CB} together imply that $(G,\vv\kk)$ is \emph{not} a disguised toric dynamical system.
    
    Lastly, we argue that condition \ref{proj-counter-ex::chamberneq4} implies $(P(G), \vv\kk)$ is a disguised toric dynamical system. Notice that the net reaction vectors of $\yy'_1$ and $\yy'_2$ are 
    \eqn{ \label{eq:net-rxn-vect-ex_w12} 
        \vv w'_1 = \left(  \frac{1}{7} \kk_{12} + \kk_{13} - \kk_{14} \right) \begin{pmatrix*}[r] -1 \\ 1 \end{pmatrix*} 
        \quad \text{and} \quad 
        \vv w'_2 = \frac{1}{7} \left( \kk_{21} + 8 \kk_{24}  - 6\kk_{23} \right)  \begin{pmatrix*}[r] 1 \\ -1 \end{pmatrix*}
    }
    respectively. Condition \ref{proj-counter-ex::chamberneq4} says that either $\vv w'_2$ points towards $\yy'_1$ or $\vv w'_1$ points towards $\yy'_2$. In other words, satisfying \ref{proj-counter-ex::chamberneq4} puts $(P(G),\vv\kk)$ in the analog of Chambers 1--3 in the language of \cite{BCS22}.

    We now show that $(P(G),\vv\kk)$ is always disguised toric. Suppose we are in the case of $6\kk_{23} < \kk_{21} + 8\kk_{24} $ and $\kk_{14} < \frac{1}{7} \kk_{12} + \kk_{13}$, so that $\vv w_1'$ points towards $\yy_2'$ and $\vv w_2'$ points towards $\yy_1'$. We claim that this is dynamically equivalent to the following single-target network $(G', \vv\alpha)$. Let $\yy'_5 =  ( \frac{13}{14}, \frac{29}{14})^\top$, which lies between $\yy'_1$ and $\yy'_2$, and consider the reactions 
    \eq{ 
        \yy'_1 \xrightarrow{\,\,\alpha_1\,\,} \yy'_5, \qquad 
        \yy'_2 \xrightarrow{\,\,\alpha_2\,\,} \yy'_5,  \qquad 
        \yy'_3 \xrightarrow{\,\,\alpha_3\,\,} \yy'_5,  \qquad 
        \yy'_4 \xrightarrow{\,\,\alpha_4\,\,} \yy'_5,  
    }
    with rate constants 
    \begin{gather*} 
        \alpha_1 = 14\left(\frac{1}{7}\kk_{12}+\kk_{13}-\kk_{14}\right), \quad 
        \alpha_3 = \frac{14}{13} \left(\frac{6}{7}\kk_{32} + \kk_{31} + 2\kk_{34}\right), \\ 
        \alpha_2 = 2\left(\kk_{21}+8\kk_{24} - 6\kk_{23}\right),   
         \quad 
        \alpha_4 = \frac{14}{15}\left(\frac{8}{7}\kk_{42}+\kk_{41}+2\kk_{43}\right).
    \end{gather*}
    Note that $\alpha_i > 0$, and are chosen so that the net reaction vectors are \eqref{eq:net-rxn-vect-ex_w12} and 
    \eq{ 
        \vv w_3' = \left( \frac{6}{7}\kk_{32} + \kk_{31} + 2\kk_{34} \right)  \begin{pmatrix*}[r] 1 \\ -1 \end{pmatrix*} 
        \quad \text{and} \quad 
        \vv w_4' = \left(\frac{8}{7}\kk_{42} + \kk_{41} + 2 \kk_{43} \right)  \begin{pmatrix*}[r] -1 \\ 1 \end{pmatrix*}. 
    }
    In other words, the rate constants $\alpha_i$ are chosen so that $(P(G),\vv\kk)$ is dynamically equivalent to the single-target network $(G', \vv\alpha)$. Single-target networks were characterized in \cite{CraciunJinYu_STN}; in particular, a single-target network is dynamically equivalent to a detailed-balanced system---hence disguised toric---if and only if the unique sink is in the relative interior of the convex hull of source vertices. This is clearly the case for $(G',\vv\alpha)$, as $\yy'_5$ lies midway between $\yy'_1$ and $\yy'_2$. Therefore, $(P(G),\vv\kk)$ is disguised toric. It is not difficult to show that if only one of the inequalities in condition \ref{proj-counter-ex::chamberneq4} is satisfied, then the mass-action system is still dynamically equivalent to a stable single-target network, with the sink placed appropriately. 
    
    In particular, the invertible transformation $P^{-1}$ takes the disguised toric dynamical system $(P(G),\vv\kk)$ to the dynamical system $(G,\vv\kk)$ which is not disguised toric. In other words, the property of being disguised toric is in general not preserved under projective transformation.
\end{ex}

\begin{ex}[\textbf{Projective transformations do not preserve complex-balancing}] 
\label{ex:CB-not_preserved-proj}
    Consider the complete graph $G$ embedded in $\rr^1$ with vertices at $0$, $1$, $2$. Let $P$ be an invertible projective transformation defined at $0,1,2$, and let $a=P(0)$, $b=P(1)$, $c=P(2)$. The complex balanced condition on $(G,\vv\kk)$ is equivalent to  $K_1K_3=K_2^2$, where the $K_i$ are the tree constants as given by the Matrix-Tree Theorem (for a definition, see \cite[page 5]{CraciunDickensteinShiuSturmfels2009}); each $K_i$ is a sum of monomials in the $\kk_{j\ell}$ determined by the network connectivity. On the other hand, since $(\yy_3 - \yy_2)/(c-b) = (\yy_2-\yy_1)/(b-a)$, by \cite[Theorem 2]{FeliuCappellettiWiuf2018} the complex-balanced condition on $(P(G),\vv\kk)$ is  
    \eq{ 
        K_1^{c-b}K_3^{b-a}=K_2^{c-a}.
    }
    When $a+c\neq 2b$, which holds for a generic choice of $P$, the algebraic sets $K_1^{c-b}K_3^{b-a}=K_2^{c-a}$ and $K_1K_3=K_2^2$ are not equal, and so for generic $\vv\kk$ with $(G,\vv\kk)$ complex-balanced, we see $(P(G),\vv\kk)$ is not complex-balanced.
\end{ex}

\begin{ex}[\textbf{Projective transformations do not preserve dynamical equivalence}] 
\label{ex:projective-not-->DE}
    Consider the complete graph $G$ embedded in $\rr^1$ with vertices at $0$, $1$, $2$. Say that $(G,\vv\kk)$ and $(G,\vv\alpha)$ are dynamically equivalent but $\vv\kk\neq\vv\alpha$. Letting $\beta_{ij} \coloneqq \kk_{ij}-\alpha_{ij}$ and
    $\vv \beta=\begin{pmatrix}
    \beta_{01} & 
    \beta_{02} & 
    \beta_{12} & 
    \beta_{10} & 
    \beta_{20} & 
    \beta_{21}
    \end{pmatrix}^\top$,
    then $(G,\vv\kk)$ and $(G,\vv\alpha)$ are dynamically equivalent if and only if  $\vv \beta\neq \vv 0$ and 
    \eqn{\label{eqn:P1-DE1}
        \begin{pmatrix}
        1 & 2 & 0 & 0 & 0 & 0\\
        0 & 0 & 1 & -1 & 0 & 0\\
        0 & 0 & 0 & 0 & -2 & -1
        \end{pmatrix}
        \vv \beta=\vv 0.
    }
    
    Let $P$ be an invertible projective transformation defined at $0,1,2$, and let $a=P(0)$, $b=P(1)$, $c=P(2)$. Further suppose that $a+c\neq 2b$, which holds for a generic invertible projective transformation.

    Note that $(P(G),\vv\kk)$ and $(P(G),\vv\alpha)$ are dynamically equivalent if and only if
    \eqn{\label{eqn:P1-DE2}
        \begin{pmatrix}
        b-a & c-a & 0 & 0 & 0 & 0\\
       0 & 0 & c-b & a-b & 0 & 0 \\
        0 & 0 & 0 & 0 & a-c & b-c
        \end{pmatrix}
        \vv\beta=\vv0.
    }
    Since $a$, $b$, $c$ are distinct, the non-zero vector $\vv\beta$ satisfies both \eqref{eqn:P1-DE1} and \eqref{eqn:P1-DE2} if and only if  $c-a=2(b-a)$, $c-b=-(a-b)$, and $a-c=2(b-c)$; this is equivalent to the condition that $a+c=2b$. Thus, $(P(G),\vv\kk)$ and $(P(G),\vv\alpha)$ are no longer dynamically equivalent. This implies that equivalence classes of dynamically equivalent mass-action systems are not preserved under invertible projective transformations. 
\end{ex}

\bibliographystyle{siam}
\bibliography{cit}
\end{document}